\documentclass[11pt]{article}%
\usepackage{amsmath}
\usepackage{amsfonts}
\usepackage{amssymb}
\usepackage{graphicx}
\graphicspath{{images/}}
\usepackage{color}
\usepackage{geometry}%
\usepackage{multirow,booktabs}
\usepackage{subfigure}
\usepackage{epstopdf}
\usepackage{mathrsfs}
\usepackage{float}
\usepackage{lineno,hyperref}
\usepackage{latexsym}
\usepackage{cite}
\usepackage[sort&compress]{natbib}
\usepackage{mathtools}
\providecommand{\U}[1]{\protect\rule{.1in}{.1in}}

\newtheorem{theorem}{Theorem}[section]

\newtheorem{definition}{Definition}[section]

\newtheorem{remark}{Remark}[section]

\newtheorem{assumption}{Assumption}[section]
\newenvironment{proof}[1][Proof]{\noindent\textbf{#1.} }{\ \rule{0.5em}{0.5em}}
\numberwithin{equation}{section}

\newcommand{\be}{\begin{equation}}
	\newcommand{\ee}{\end{equation}}

\newcommand{\bq}{\begin{eqnarray}}
	\newcommand{\eq}{\end{eqnarray}}

\allowdisplaybreaks

\begin{document}

\title{Optimal investment problem for a hybrid pension with intergenerational risk-sharing and longevity trend under model uncertainty}






\author{Ke Fu${}^{\mbox{\tiny a}}$,  \quad Ximin Rong${}^{\mbox{\tiny a,\tiny b}}$	
	\, , Hui Zhao${}^{\mbox{\tiny a},}$\thanks{Corresponding author. Email address: zhaohuimath@tju.edu.cn (H. Zhao).}\\
	${}^{\mbox{\tiny a}}$School of Mathematics, Tianjin University, Tianjin 300350, P.R. China\\
${}^{\mbox{\tiny b}}$Center for Applied Mathematics, Tianjin University, Tianjin 300072, P.R. China}

\date{}
\maketitle
\baselineskip15.6pt

\begin{center} {\bf Abstract} \end{center}
This paper studies the optimal investment problem for a hybrid pension plan under model uncertainty, where both the contribution and the benefit are adjusted depending on the performance of the plan. Furthermore, an age and time-dependent force of mortality and a linear maximum age are considered to capture the longevity trend. Suppose that the plan manager is ambiguity averse and is allowed to invest in a risk-free asset and a stock. The plan manager aims to find optimal investment strategies and optimal intergenerational risk-sharing arrangements by minimizing the cost of unstable contribution risk, the cost of unstable benefit risk and discontinuity risk under the worst-case scenario. By applying the stochastic optimal control approach, closed-form solutions are derived under a penalized quadratic cost function. Through numerical analysis and three special cases, we find that the intergeneration risk-sharing is achieved in our collective hybrid pension plan effectively. And it also shows that when people live longer, postponing the retirement seems a feasible way to alleviate the stress of the aging problem.


\vskip 3mm

\section{Introduction}

Recently, the number of retirees is increasing because of the increasing life expectancy, and there are great pressures on the pension systems due to the change of the economic and demographic.  Defined benefit (DB) plans and defined contribution (DC) plans are two traditional types of pension scheme. In a DB plan, plan trustees bear the financial risk while in a DC plan, the participants bear all the longevity and investment risks since they hold their own individual accounts. Furthermore, DC plans are affordable and transparent, as the participants have full control over their contributions. However, they may not obtain adequate benefits upon retirement because their benefits are somewhat depending on the asset returns and interest rates, which could be volatile and unpredictable. On the other hand, DB tends to provide adequate benefits as predetermined. Therefore, they are generally opaque and not necessarily afforable, especially from the plan trustees' perspective. Thus, neither of these plans is ideal to face the challenges of the demographic transition. The social partners have been planning to revise the regulatory framework for pension funds. An ideal pension scheme should provide adequate benefits and risk-sharing between participants and plan trustees.

Recently, many researchers propose the hybrid pension plan which combines the advantages of DB and DC plans. Hybrid pension plan is in the middle of the two "extremes" and can be more flexible. Therefore, many countries begin to implement various hybrid pension plans. For example, the collective DC plans (\cite{2013The}; \cite{2014The}), target benefit plans (\cite{2013New}; \cite{Wang2018Optimal}), risk-sharing DB plans (\cite{Pugh2008Funding}), the floor-offset plans ( \cite{Kai2009The}) and Notional Defined Contribution (NDC) plans (\cite{2016Optimal}, \cite{Jennifer2017Automatic}). \cite{2014Hybrid} provide an overview of the development of hybrid pension plans around the world and classify the plans into four different types: the hybrid DB plans in the Netherlands, the nonfinancial DC plan in Sweden, cash balance plans in the United States, Canada and Japan, and the Riester plans in Germany. 

The hybrid pension plan aims to provide a better retirement security on a sustainable, stable and affordable basis with risk sharing among different age cohorts. Actually, it is well documented that intergenerational risk-sharing may be welfare enhancing. \cite{2008Intergenerational} proves that intergenerational risk transfer is welfare-improving for all current and future generations in a collective pension scheme compared with participants in individual schemes. This is because a better intergenerational risk-sharing scheme makes it socially efficient to raise the collective risk exposure in order to take advantage of the large equity premium. \cite{2011Intergenerational} examines a funded system with a realistic description of risks in financial markets. Comparing risk sharing in various types of funded pension systems such as individual DC plans, collective DC plans and traditional DB plans, they find that risk-shifting over time (due to adjustment mechanisms) leads to higher utilities in pension schemes. \cite{Jean2020Levelling} studies intergenerational risk and cost sharing for a variety of collective funded pension plans with time-varying contributions and benefit levels considering the Volatility Index (VIX). He also concludes that pension schemes with a well-structured volatility-risk-adjusted component can be welfare enhancing for the entry and future cohorts. Similar researches can be found in \cite{Teulings2006Generational}, \cite{2012Voluntary}, \cite{Beetsma2016Intergenerational}, \cite{Boes2018Intergenerational}. Besides, there are other studies in the literature having tried to find optimal contributions and/or benefit adjustment policies within hybrid pension funds. See  \cite{2012Risk}, \cite{2019Optimal}, \cite{2020Optimal} and the references therein.

Furthermore, the issue of longevity is an important factor when designing intergenerational risk-sharing scheme. The increasement of people's average life expectancy poses a threat to the traditional pension plans. \cite{2013Long} show that  improvements in life expectancy have been driven to a large part by expansions of the maximum age from the second half of the 20th century onwards. Therefore, \cite{Knell2018Increasing} thinks that a straightforward way to capture increasing life expectancy is to assume that the maximum age increases in a linear fashion according to time. \cite{2015OPTIMAL} introduce the factor that is a random effect on the expected survival probability to reflect systematic longevity risk in a two-period model. Besides, it captures the longevity trend through the force of mortality that is not only age-dependent but also include a cohort-specific time. To be more specific, in the long run, people's force of mortality is relatively low due to the progress of medical treatment and development of economy. On the other hand, it is obvious that the force of mortality raises when people's age increases. Furthermore, \cite{Knell2018Increasing}, \cite{2019Continuous}, \cite{Andr2014On} and \cite{2017NDC} all take the socio-economic status and income into consideration in model of mortality. Inspired by \cite{Knell2018Increasing}, we assume that the maximum age is an incresing function of the time in this paper. Moreover, we apply an a age and time-dependent force of mortality which we call modified Makeham's Law to capture the longevity.

Traditionally, the manager of the hybrid pension plan is assumed to know exactly the true probability measures of financial market. However, decision-makers are uncertain about the true model in practice. Because, for example, the parameters (especially the drift parameters) are hard to estimate with precision (\cite{1980On}). Therefore, it is fair to say that any particular probability measure used to describe the model would be subject to a considerable degree of model misspecification. This type of uncertainty caused by the lack of information about the probability measure is also referred to as ambiguity, which is apparently different from risk where the model is characterized by a single probability measure. To deal with the decison makers' ambiguity aversion, \cite{0A} proposes a robust control approach in a continuous-time framework. They assume that the decision-maker does not trust the specific probability measure. and treats it as a reference measure, and takes into account a set of alternative measures that is statistically difficult to distinguish from the reference measure. The gap between the reference measure and an alternative measure is constrained by the relative entropy, which acts as a penalty term in the optimization procedure. This penalty captures the decision-maker’s ambiguity aversion about the reference measure. Rencently, \cite{2018Robust} investigate a robust optimal investment problem for an ambiguity-averse member of defined contribution (DC) pension plans with stochastic interest rate and stochastic volatility. In addition, \cite{2021WANG} consider the optimal investment and benefit payment problem for a target benefit plan (TBP) with default risk and model uncertainty.
 
However, to the best of our knowledge, there is few literature on investment problem of hybrid pension plan under ambiguity. But the plan manager is always uncertain about the reference model. Therefore we assume that a representative hybrid plan member is ambiguity averse in this paper. Suppose that all the participants are in an aggregate pension fund with both active members and retirees in an overlapping generations' economy. Contributions are paid by active members and at the same time the retired members receive benefits from the pension fund. It is more reasonable that all generations should share the risk and this is achieved by adjusting contributions and benefits as a function of the fund surplus which is the difference between the asset and target liability. This adjustment scheme is motivatied by real-world example (see e.g.\cite{2009Sharing}). In this collective hybrid pension scheme, we focus on the strategic portfolio allocation decisions together with pension policies, that are intergenerational risk sharing rules, to dynamically adjust financial position of the pension funds. Inspired by \cite{2019Optimal}, the ultimate goals of this study is to provide retirees adequate and stable retirement benefits and to ensure fair and effective contributions from active members. Meanwhile, plan manager is ambiguity averse and aims to keep the system financially sustainable ensuring the generational equity under the uncertainty. Mathematically, these goals are achieved through minimizing the expected utility of intergenerational risk, adjustment of the contribution and the benefit, which is called the discontinuity risk, the unstable contribution risk and the unstable benefit risk, respectively, by choosing appropriate cost function under the worst-case scenario. According to \cite{1997Stochastic}, \cite{2001Minimization} and \cite{2007Stochastic}, we consider the quadratic cost function.

The contributions of this paper is as follows. First of all, we take model uncertainty into account, which reflect the fact that in many situation decision-makers are uncertain about the true model, and try to find the robust strategy for the pension plan under the worst-case scenario. Then, the model we considered, comparing to \cite{2011Intergenerational} and \cite{2012Risk}, is in an overlapping generations' economy (OLG). This OLG model allows people to enter and leave at every continous time and all the generations will share the risk together. Besides, we consider the age and time-dependent force of mortality that is the modified Makeham's Law, which reflects the longevity trend of population. At last, followed by \cite{2020Optimal}, we consider the negative entrant growth rate of the labor supply, which could reflect the fact of low fertility. 

The rest of the paper is structured as follows. In Section 2, we introduce the model of financial market and the hybrid pension plan. In Section 3, we propose the optimal investment problem for the hybrid pension and robust strategy under model uncertainty. In Section 4, we give three special cases. In Section 5, we present some numerical results and sensitivity analysis. Finally, Section 6 concludes the paper and provide some further research ideas.

\section{Model formulation}
Let $(\Omega,\mathscr{F}, \mathbb{P})$ be a complete probability space, where $\mathscr{F}=\{\mathscr{F}_t\}_{t\geq0}$ is a complete and right continuous filtration generated by a two-dimensional standard Brownian Motion, that is, $\mathscr{F}_t$ is the filtration containing the information about the financial market and the information about the salary of the members at their time of retirement at time $t$, which is available to the sponsor of the pension plan, and $\mathbb{P}$ is a probability measure on $\Omega$.

\subsection{The financial market} 
The financial market consists of a risk-free asset (bond) and a risky asset (stocks). The price of the risk-free asset and risky asset at time $t\geq0$ is described by
\begin{align}
	&\textrm{d}S(t)=rS(t)\textrm{d}t, \quad  S(0)=1,\label{risk-free asset} \\
	&\textrm{d}S_{1}(t)=S_{1}(t)\left(\mu \textrm{d}t+\sigma \textrm{d}B(t)\right), \quad    S_{1}(0)=1,  \label{risky asset}
\end{align}
where $r>0$ denotes the risk-free interest rate. $\mu$ is the appreciation rate of the stock and $\sigma$ is the volatility rate, both $\mu$ and $\sigma$ being positive constants, and $B(t)$ is a standerd Brownian motion. To exclude arbitrage opportunities, we assume that $\mu>r$.

\subsection{Population dynamics}
For population, We will develop a continuous deterministic framework considering that age and time are continous variables. Plan members are assumed to enter the labor market at age $x_0$ and to keep active up to a retirement age $x_r$. As above mentioned, it is reasonable that the force of mortality decrease with time $t$ and increase with age $x$ because of the decreasing fertility and increasing life expectancy. Here, we consider an age and time-dependent mortality $\mu(x,t)$ followed by modified Makeham's Law, which may capture longevity trend dynamically as following assumption.
\begin{assumption}
	The age and time-dependent force of mortality $\mu(x,t)$ is denoted by 
	$$\mu(x,t)=\mathcal{A}+\mathcal{B} \theta^{x-\frac{1}{\omega} t},$$
\end{assumption}

where $\omega$ is a positive parameter called longevity parameter which shows longevity trend that is every $\omega$ year, the average life expectancy increases by one year. More precisely, for a $x$-year-old cohort at time 0, the force of mortality is $\mu(x,0)$. $\omega$ years later, the force of mortality at time $\omega$ is $\mu(x+\omega,\omega)=\mathcal{A}+\mathcal{B} \theta^{x+\omega-1}$. And theoretically, if there is no longevity parameter, the force of mortality at time $\omega$ could be like  $\mathcal{A}+\mathcal{B} \theta^{x+\omega}$. Therefore the model we considered shows that people have relatively low force of mortality and is equivalent to be one-year ``younger''. On the whole, the average life expectancy would increases by one year.

\begin{remark}
In this age and time-dependent force of mortality, we note that the higher $\omega$, the trend of longevity comes more slowly which means it takes long time for the whole people be long-lived. Instead, the lower $\omega$, the aging problem would be more urgent. For example, according to \cite{European Commission(2015)}, for the European Union(EU)-countries, life expectancy at birth is projected to increase over the next 50 years by about 7.5 years, which means almost $\omega=6.67$.
\end{remark}

\begin{assumption}
The maximum age is denoted by $m(t)$  
$$ m(t)=m_0+\xi t, 1 \geq \xi \ge 0$$
where $m_0$ is the initial maximum survival age.

\end{assumption}

This assumption shows that the maximum age increases in a linear fashion and is in line with the empirical literature. \cite{Knell2018Increasing} and \cite{Oeppen2002} e.g., analyze `record female life expectancy' (i.e., the highest value for female life expectancy reported in any country for which data are available) from 1840 to 2000 and they show that it follows an almost perfect linear development. Actually, \cite{2013Long} have shown that from the second half of the 20th century onwards improvements in life expectancy have been driven to a large part by expansions of the maximum age.


Denote $p(x,t)$ the survival function that the cohort who enter the system at time $t-(x-x_0)$ survives to age $x$. It holds that $p(x_0, t)=1, p(m(t),t) = 0$ and that survivorship declines with age, i.e., $\textrm{d}p(x,t)/\textrm{d}x \leq 0$ for $x \in [x_0, m(t)]$. The mortality hazard rate of cohort $t$ at age $x$ is given by 
$$\mu(x,t)\equiv -\frac{\textrm{d}p(x,t)}{\textrm{d}x}\frac{1}{p(x,t)},$$
therefore, $p(x,t)$ and $\mu(x,t)$ are not only age-specific but also include a cohort-specific time index. It holds that:
$$p(x,t)={\rm e}^{- \int_{0}^{x} \mu(x,t) \textrm{d}x},$$
Thus the age and time-dependent survival function $p(x,t)$ can be expressed as
\begin{equation}
	\begin{split}
		p(x,t)&={\rm e}^{-\int_{0}^{x-x_0}{\mu(x_0+u,t-(x-x_0)+u)} \textrm{d}u}\\
		&={\rm e}^{-\mathcal{A}(x-x_0)-\frac{\mathcal{B}}{ln\theta} (\theta^{ x-\frac{1}{\omega} t}-\theta^{(1 -\frac{1}{\omega})x_0+\frac{1}{\omega} (x-t)})}, \qquad t\ge 0, m(t)\geq x \ge x_0.
	\end{split}
\end{equation}

We denote by $n(t)$ the entry rate which describes the density of sizes of new cohort entering in working life at age $x_0$ at time $t$. $n(t)$ follows Malthusian demographic model $n(t)=n_0e^{\kappa t}$, where $n_0$ is the density of the new $x_0$-year-old entrants at time $0$, $\kappa$ is the growth rate of the new entrants. Inspired by \cite{2020Optimal} and \cite{he}, $\kappa<0$ represents the demographic model with serious aging problem, which is caused by low fertility and low mortality. Then the number of those attaining age $x$ at time $t$ is $$n(t-(x-x_0))p(x,t), \qquad x>x_0,$$
where $t-(x-x_0)$ is the time at which this cohort joined the pension system. Note that with $t\ge 0 $, the value of $t-(x-x_0)$ could be negative, meaning that a member whose age is $x$ at time $t\ge 0$ joined the pension system $(x-x_0)$ years ago. The total number of active member who make contributions at time $t$ is given by
 \begin{equation}
NC(t)=\int_{x_0}^{x_r} {n(t-(x-x_0))p(x,t)} \textrm{d}x, \label{NC}
\end{equation}
and the total number of retired members who recieve the benefits at time $t$ is
\begin{equation}
	NB(t)=\int_{x_r}^{m(t)} {n(t-(x-x_0))p(x,t)} \textrm{d}x. \label{NB}
\end{equation}

Here, we assume that during the working period, each cohort receives his labor income, which is normalized to 1. Therefore the total contribution income $TC=\{TC(t):t\ge 0\}$ is
\begin{equation}
	TC(t)=\int_{x_0}^{x_r} {n(t-(x-x_0))p(x,t)c(t)} \textrm{d}x=NC(t)c(t), \label{overline TC}
\end{equation}
and the total benefit payment $TB=\{TB(t):t\ge 0\}$ is
$$TB(t)=\int_{x_r}^{m(t)} {n(t-(x-x_0))p(x,t)b(t)} \textrm{d}x=NB(t)b(t),$$
where $c(t)$ and $b(t)$ is the amount of the contribution and benefit at time t, respectively. The specific forms about $c(t)$ and $b(t)$ will be given by the next section.

\begin{remark}
Equation (\ref{NC}) and (\ref{NB}) can be observed in some particular cases. For instance, the model analyzed in \cite{Wang2018Optimal} can be recovered by setting $\kappa=0$, which means the entry rate is constant $n(t)=n_0$. And $\omega=\infty$  means the force of mortality is only age-dependent, which is in line with the \cite{2020Optimal}.	
\end{remark}

The surplus (deficit) between the total contribution income and the total benefit payment will increase (decrease) the accumulation of the pension fund. Meanwhile, the accumulation is dynamically allocated to a risk-free asset and a risky asset.


\subsection{Hybrid pension scheme} 

The pension fund trustee is responsible for the asset-liability management. The amount of money invested in the risky asset at time $t$ is denoted by $\pi=\{\pi(t),t \ge 0\}$ and then the amount invested in risk-free asset is $A(t)-\pi(t)$. The dynamics of the asset process $A(t)$ can be described by 
\begin{equation}
	\textrm{d}A(t)=\pi(t)\frac{\textrm{d}S_1(t)}{S_1(t)}+[A(t)-\pi(t)]\frac{\textrm{d}S(t)}{S(t)}+[TC(t)-TB(t)]\textrm{d}t \label{asset process 1}, \quad 
	A(0)=a_0,
\end{equation}
where $a_0$ is the initial asset at time 0. Using (\ref{risk-free asset}) and (\ref{risky asset}), we can easily rewrite (\ref{asset process 1}) as
\begin{equation}
	\textrm{d}A(t)=[\pi(\mu-r)+A(t)r]\textrm{d}t+[NC(t)c(t)-NB(t)b(t)]\textrm{d}t+\sigma\pi\textrm{d}B(t), \quad
	 A(0)=a_0. \label{asset}
\end{equation}

Consider the target contribution $c$ and target benefit $b$ at time $0$ and a finite time interval $[0,T]$ with $T$ being the so-called terminal time. For each age cohort, denoted by $x$, the target liability equals the difference between the present value of risk-free benefits and the present value of the yet-to-be-paid risk-free contributions. This is in line with the one in the literature of \cite{2011Intergenerational}.
\begin{equation}
	L(x)=\begin{cases}
		\int_{0}^{T} {(NB(t)b{\rm e}^{\tau t}-NC(t)c{\rm e}^{\tau t}){\rm e}^{-r(t-(x-x_0))}} \textrm{d}t,& for \quad x_r\geq x \ge x_0;\\
		\int_{0}^{T} {NB(t)b{\rm e}^{\tau t}{\rm e}^{-r(t-(x-x_0))}} \textrm{d}t,& for\quad m(t)\geq x \ge x_r,	
	\end{cases}
\end{equation}
where $\tau$ is the target instantaneous growth rate of the target contribution and the target benefit.

At the aggregate level, the target liabilities of the fund can be calculated simply as the sum of the liabilities for each age cohort.
\begin{equation}
	\begin{split}
	L=&\int_{x_0}^{x_r} {\int_{0}^{T} {(NB(t)b{\rm e}^{\tau t}-NC(t)c{\rm e}^{\tau t}){\rm e}^{-r(t-(x-x_0))}} \textrm{d}t} \textrm{d}x \\
		+&\int_{x_r}^{m(t)} {\int_{0}^{T} {NB(t)b{\rm e}^{\tau t}{\rm e}^{-r(t-(x-x_0))}} \textrm{d}t} \textrm{d}x. \label{target liabilties}
	\end{split}
\end{equation}
Given the sationary age composition of the fund and the fixed target benefit level, the target liability $L$ is time invariant. Actually, if the fund invests fully in the risk-free asset, then the actual liability follows exactly as the target liability $L$. If the fund accepts mismatch risk, for example by investing in stocks, there may be funding surpluses or deficits with respect to the target liability.

The fund surplus is defined as the diference between assets and the target liability level
$$SP(t)=A(t)-L.$$
Notice that the collective pension schemes may leave surpluses or deficits which lead to intergenerational transfers and hence to intergenerational risk sharing.

Besides the adjustment of the plan members' traget liabilities as described in (\ref{target liabilties}) by linking to the fund assets as decribed in (\ref{asset}), for the collective hybrid pension model, we also dynamically and simultaneously adjust both the current individual contribution paid by the active members and individual benefit payments payable to the retirees, implicitly influenced by the surplus level of the pension fund in order to make intergenerational risk sharing.

To achieve our goal mentioned above, we specify a simple cohort contribution $c(t)$ and benefit $b(t)$ policy, where the contribution per cohort is a function of the target contribution rate level $c\rm{e}^{\tau t}$ and the funding residual per active cohort $\frac{SP(t)}{NC(t)}$, and the benefit per cohort is a function of the target benefit rate level $b\rm{e}^{\tau t}$ and the funding residual per retire cohort $\frac{SP(t)}{NB(t)}$.
\begin{gather}
	c(t)=c{\rm e}^{\tau t}-\alpha \frac{SP(t)}{NC(t)}=c{\rm e}^{\tau t}-\lambda_1(t),  \label{c(t)} \\
	b(t)=b{\rm e}^{\tau t}+\beta \frac{SP(t)}{NB(t)}=b{\rm e}^{\tau t}+\lambda_2(t), \label{b(t)}
\end{gather}
where $\lambda_1(t)=\alpha \frac{A(t)-L}{NC(t)}$, $\lambda_2(t)=\beta \frac{A(t)-L}{NB(t)}$ which represent the adjustments to the pension contributions and benefit payments, respectively. The slope coefficient $\alpha$ and $\beta$ change over time according to the adjustment of the hybrid pension system thus reflect the speed of absorbing the funding imbanlances, being the so-called spread parameters for contribution income and benefit payment, respectively. $\alpha=1$ or $\beta=1$ implies that a funding imbalance is immediately and fully abdorbed. More specifically, a fraction $\alpha$ of the funding residual is shared among employees and a fraction $\beta$ of the funding residual is shared among retirees. A lower vaule of $\alpha$ and $\beta$ imply that part of the funding surplus is shifted to the future, and shared across generations i.e. the lower vaule of $\alpha$ and $\beta$, the higher the degree of intergenerational risk sharing.

The two adjusted processes $\{c(t);t\ge 0\}$ and $\{b(t);t\ge 0\}$ describe the pension system over the time which consist of decision on contribution and benefit payment for each cohort at time $t$. When the underlying fund value at time t has surplus ($SP(t)>0$), indicating an out performance in assets investment, we expect both positive policy adjustment of $\lambda_1(t)$ and $\lambda_2(t)$ so that at time t the benefit payments would be increased while the contributions required being decreased. When the underlying fund incurs a deficit due to investment losses ($SP(t)<0$), we in turn expect both negative policy adjustment of $\lambda_1(t)$ and $\lambda_2(t)$. Here, increasing the contributions and decreasing the benefits can re-balanced the budget. All in all, high and low asset returns of the fund are spread over active members and retirees, which implies an intergenerational risk transfer or sharing. Such hybrid pension model has a similar mathematical structure as the hybrid model studied in \cite{2011Intergenerational}, \cite{2012Risk} and \cite{2019Optimal}.

In the following section we present, for the pension model described above, we set a continuous-time robust stochastic optimal control problem for the hybrid pension plan.

\section{Optimal investment problem for the hybrid pension under model uncertainty}
\subsection{Optimization problem}
In the traditional framework of pension plan or investment problem, the member
is assumed to be ambiguity-neutral, which implies that the manager has
complete confidence in the above model provided by probability
measure $\mathbb{P}$. However, the fact is that in many cases the member
cannot know exactly the true model, and thus any particular probability
measure used to describe the model would lead to potential
model misspecification. The parameters
of financial markets are difficult to estimate with
precision in practice and numerous experimental studies demonstrate
that the investors are averse not only to risk (the known
probability distribution) but also to ambiguity (the unknown
probability distribution) (refer to \cite{0Ambiguity}). Therefore, we assume
that the manager is ambiguity-averse and aims to protect himself
against the worst-case scenario. The manager recognizes that
the models under the reference probability $\mathbb{P}$ only approximate
the true models and takes into account some alternative models,
which can be defined via a family of probability measures
equivalent to $\mathbb{P}$ as follows
$$\{\mathscr{Q}:={\mathbb{Q}\arrowvert\mathbb{Q}\thicksim \mathbb{P}}\}.$$

Based on this assumption, we can use Girsanov’s theorem to
change the probability measures, so that the models defined by
the alternative measures only differ in terms of drift function. By
Girsanov’s theorem, for each $\mathbb{Q} \in \mathscr{Q}$, there is a progressively measurable process $\phi(t),t\in [0,T]$, which can be
referred as the probability distortion process, such that
$$\dfrac{\textrm{d}\mathbb{Q}}{\textrm{d}\mathbb{P}}|_{\mathscr{F}t}=\Lambda^{\phi}(t),$$
where
\begin{equation}
	\Lambda^{\phi}(t)=\textrm{exp}\big\{\int_{0}^{t} \phi(s)\textrm{d}B(s)-\frac{1}{2}\int_{0}^{t} \phi(s)^{2}\textrm{d}s\}.
\end{equation}
Moreover, if  $\phi(t),t\in [0,T]$ satisfies Novikov’s condition
$$\textrm{E}^{\mathbb{P}}\big\{\frac{1}{2} \int_{0}^{t}\phi(s)^{2}\textrm{d}s\}<\infty,$$
then the Radon-Nikodym derivative process $\Lambda^{\phi}(t)$ is a $\mathbb{P}$-martingale.

According to Girsanov’s theorem, under the alternative measure
$\mathbb{Q}$, the stochastic process $B^{\mathbb{Q}}(t)$ is one-dimensional standard Brownian motion, and we have
$$\textrm{d}B^{\mathbb{Q}}(t)=\textrm{d}B(t)-\phi(t)\textrm{d}t.$$
Hence, for the given measure $\mathbb{P}$, choosing an alternative measure $\mathbb{Q}$ is equivalent to determining the process $\phi$. Furthermore, the price of the stock and dynamic asset process under alternative measure $\mathbb{Q}$ become
\begin{align}
\textrm{d}S_{1}^{\mathbb{Q}}(t)&=S_{1}(t)\left(\mu \textrm{d}t+\sigma (\textrm{d}B^{\mathbb{Q}}(t)+\phi\textrm{d}t)\right), \quad    S_{1}(0)=1,\\
\textrm{d}A^{\mathbb{Q}}(t)&=[\pi(\mu-r)+A(t)r]\textrm{d}t+[NC(t)c(t)-NB(t)b(t)]\textrm{d}t+\sigma\pi(\textrm{d}B^{\mathbb{Q}}(t)+\phi\textrm{d}t).
\end{align}

\begin{definition}(Admissible Strategy). \rm For any fixed  $t \in [0,T]$, a strategy
	$$\boldsymbol{\pi} =\{(\pi(u),  \lambda_1(u)), \lambda_2(u)\}_{u\in[t,T]}$$
	is said to be admissible if
	\begin{itemize}
		\item[(i)] $\boldsymbol{\pi}$ is $\mathscr{F}_t$-adapted;
		\item[(ii)] for all $u \in [t,T]$, $f(t)\geq 0$ and $\textrm{E}\left[\int_{t}^{T}[\pi(u)]^2\textrm{d}u\right]<+\infty$;
		\item[(iii)] $(A^{\boldsymbol{\pi}}, \boldsymbol{\pi})$ is the unique solution to SDE (\ref{asset}).
	\end{itemize}
\end{definition}

Recall that the hybrid pension plan considered consists of both active members who make contributions to the fund and retired members who receive pension benefits. Since the contribution and benefit are time variant, the plan trustees must choose an asset allocation, contribution and benefit strategies which strike a reasonable compromise between the interests of these two groups. Ideally, the fund should make contributions and benefits in a stable and sustainable manner to each cohort as close to the target as possible. As the pension plan we considered is a collective one with different generations involved, taking the all generations into consideration, the goal of the pension trustees is that the pension plan can be run stably and sustainably and the investment risks are shared intergenerationally.

Mathematically, the objective of this continuous-time hybrid pension system is to seek a robust optimal strategy $\boldsymbol{\pi}$ to minimize the expected discounted cost of unstable contribution risk, unstable benefit risk and discontinuity risk over the time interval $[t,T]$. The model setting is in line with \cite{2020Optimal} and \cite{Wang2018Optimal}. In particular, the contribution and benefit should be stable and comparable with respect to target contibutionn  and target benefit, and the final accumulation should be neither too large nor too small, comparing with the initial accumulation. As such, we propose a cost function that consists of three parts. The first part is the deviation between the actual contribution and the target contribution. The target contribution $ce^{\tau t}$ is a reference level, At any time $t \in [0,T]$, when $c(t)$ is ralatively larger (smaller) than $ce^{\tau t}$, the current working cohorts (cohorts at previous times) naturally feel unstatisfied because of bearing higher contribution burden, comparing to the cohorts at previous times (current working cohorts). In this cricumastance, there is welfare cost due to the unstable contribution risk. The second part is the deviation between the actual benefit and the target benefit. As same as the first part, $b(t)$ should aproach the $be^{\tau t}$ as much as possible in order to make sure that is fair to all the cohorts. If not, in this cricumastance, there is also welfare cost due to the unstable benefit risk. The third part is the initial accumulation to depict the cost of the two risks. When the final accumulation is smaller (larger) than the initial accumulation, there exist intergenerational transfer from the current working cohorts (cohorts at other times) to the cohorts at other times (current working cohorts), In this circumstance, there is welfare cost due to the discontinuous accumulation risk. Hence, the objective is to minimize the deviations within the finite time horizon $[t,T]$. In the literature \cite{2001Minimization}, \cite{2007Stochastic} and \cite{Wang2018Optimal}, both the finite and infinite time horizon are studied, and the finite time horizon vary from years to decades. In this paper, we assumed a reasonable finite time horizon because that the pension manager pays more attention to the interests of the working cohorts in their term of mangement, which means that they do not take the interests of the unborn into consideration. Further reasons about finite time horizon can be found in \cite{2020Optimal}. Employing the quadratic deviations as the cost function, We formulate the robust optimal investment problem as follows:
\begin{equation} 
	\begin{split}
		V(t,a)&=\mathop{\textrm{min}}_{\boldsymbol{\pi}\in \Pi,\phi}\mathop{\textrm{max}}_{\mathbb{Q}^{\phi} \in \mathscr{Q}} \textrm{E}_{(\boldsymbol{\pi},\phi)}^{\mathbb{Q}} \bigg\{\int_{t}^{T} {\left[\gamma_1(c(u)-c{\rm e}^{\tau t})^2{\rm e}^{-ru}+\gamma_2(b(u)-b{\rm e}^{\tau t})^2{\rm e}^{-ru}\right]} \textrm{d}u\\
		&\quad \quad \quad \quad \quad \quad \quad \quad +\gamma_3(A(T)-a_0{\rm e}^{rT})^2{\rm e}^{-rT}+	\int_{t}^{T}{\frac{(\phi(u))^2}{2\Psi(u)}}   \textrm{d}u \bigg\},  
	\end{split}\label{J(t,a,w)}
\end{equation}
where $\Pi$ is a set of all the admissible strageies of $(\pi(t),\lambda_1(t),\lambda_2(t))$. $\gamma_1$,$\gamma_2$ and $\gamma_3$ are nonnegative constants, interpreted as the weight parameter which measures the importance of the cost of unstable contribution risk, the cost of the unstable benefit risk and the cost of the discontinuous risk in the overall cost function, respectively. The expectation is computed under the alternative measure $\mathbb{Q}$ defined by $\phi$. The last part in (\ref{J(t,a,w)}) is the penalty term depending on the relative entropy which measures the discrepancy between the alternative measure and the reference measure, and the increase in the relative entropy from $t$ to $t + \textrm{d}t$ equals $\frac{1}{2}\phi^{2}\textrm{d}t$. This integral term is scaled by $\Psi(t)$ which is nonnegative and stand for the preference parameters for ambiguity aversion, measuring the degree of confidence to the reference probability $\mathbb{P}$. The larger $\Psi(t)$ is, the more ambiguity-averse the insurer is, the smaller the penalty for deviating from the reference measure is, and the less faith in the reference measure the manager has. Finally, when $\Psi=0$, the manager is completely convinced that the true measure is the reference measure $\mathbb{P}$. In this case problem (\ref{J(t,a,w)}) degenerates to the classical the quadratic deviations minization in the absence of ambiguity. When $\Psi=\infty$, the penalty term vanishes, which implies that the member is extremely ambiguous (see \cite{2003Model}; \cite{2015Robust};\cite{2018Robust}). And for analytical tractability, we assume that $\Psi(t)=k/V$, where $k$ is positive constants and stand for the ambiguity aversion parameter, meaning manager's the different levels of ambiguity aversion about the stock’s price.

Thus, the continuous-time robust stochastic optimal control problem for the hybrid pension plan is established. By choosing strategies, the manager seeks to minimize the penalized objective function in the worst-case
scenario. In the section below we aim at finding the optimal policy which solves the optimal control problem.

\subsection{Robust strategy to the optimization problem}
In this section, we use standard methods to solve the optimal control problem (\ref{J(t,a,w)}) and derive closed-form expressions of the optimal policy, denoted by $(\phi^*,\pi^*,\lambda_1^*,\lambda_2^*)$, within all the admissible policies.

First, we define a variational operator $\mathscr{A}^{\boldsymbol{\pi},\phi}$
\begin{equation}
	\begin{split}
		\mathscr{A}^{\boldsymbol{\pi},\phi}V(t,a)=& V_t+\left[\pi(\mu-r)+ar+NC(t)(c{\rm e}^{\tau t}-\lambda_1)-NB(t)(b{\rm e}^{\tau t}+\lambda_2)+\sigma\pi\phi\right]V_a\\
		&+\frac{1}{2}\pi^2\sigma^2V_{aa}+\left(\gamma_1 \lambda_1^2+\gamma_2 \lambda_2^2\right){\rm e}^{-rt},
	\end{split}
\end{equation}
where $V_t$, $V_{a}$, $V_{aa}$ are partial derivatives of $V(t,a)$.

According to the principle of stochastic dynamic programming,
the Hamilton-Jacobi-Bellman (HJB) equation can be derived as(see \cite{2003Model}; \cite{2015Robust}):
\begin{equation}
	\mathop{\textrm{min}}_{\boldsymbol{\pi}\in \Pi,\phi} \mathop{\textrm{max}}_{\mathbb{Q}} \bigg\{\mathscr{A}^{\boldsymbol{\pi},\phi}V(t,a)+ \frac{(\phi(t))^2}{2\Psi(t)} \bigg\} =0. \label{eqn:HJB}
\end{equation}
Based on the above setting, we derive a solution to the HJB
equation (\ref{eqn:HJB}) with $\Psi(t)=k/V$ . We state our findings on the optimal asset allocation, contribution and benefit adjustment policies for robust optimal control problem (\ref{J(t,a,w)}) in the following theorem. For notation simplicity, we let $\varphi=(\mu-r)/\sigma$, which is the Sharpe ratio of the risky asset.

\begin{theorem} \label{thm:main} \rm For the optimal control problem (\ref{J(t,a,w)}), the optimal asset allocation policy, contribution and benefit adjustment policy are given, respectively, by		
\begin{align}
\phi^{*}&=-\frac{k\varphi}{2(2k-1)},\\
\pi^*(t,a)&=\frac{\varphi}{(2k-1)\sigma}(a+Q(t)), \label{pi} \\
\lambda_1^*(t,a)&=\frac{\gamma_3}{\gamma_1}NC(t)P(t)(a+Q(t) ), \label{lambda_1}  \\
\lambda_2^*(t,a)&=\frac{\gamma_3}{\gamma_2}NB(t)P(t)(a+Q(t) ),   \label{lambda_2}
\end{align}
where $NC(t)$ and $NB(t)$ is show as (\ref{NC}) and (\ref{NB}) respectively. and the corresponding value function is given by
\begin{equation}
 V(t,a)=\gamma_3{\rm e}^{-rt}P(t)(a+Q(t))^{2},
\end{equation} 
where the expressions of $P(t)$ and $Q(t)$ are given below. We denote that

\begin{equation}\label{g_1 g_2 g_3}
g_{1}=\frac{NC(t)^{2}}{4\gamma_1\rm{e}^{-rt}}+\frac{NB(t)^{2}}{4\gamma_2\rm{e}^{-rt}};\quad \quad 
g_{2}=NC(t)c{\rm e}^{\tau t}-NB(t)b{\rm e}^{\tau t};\quad \quad
g_{3}=\frac{\varphi^{2}}{2k-1},
\end{equation} 

then $P(t)$ and $Q(t)$ can be show as 
\begin{equation}
	\begin{split}	
		P(t)&=\frac{1}{R(t)},\\
		Q(t)&=\frac{g_{2}}{r}(1-{\rm e}^{-r(T-t)})-a_{0}{\rm e}^{rt},
	\end{split}
\end{equation}
where
\begin{equation}
	\begin{split}
		R(t)&={\rm e}^{-(r+g_{3})(T-t)}+\frac{4g_{1}\gamma_3}{2r+g_{3}}({\rm e}^{-rt}-{\rm e}^{-(2r+g_{3})T+(r+g_{3})t}).
	\end{split}
\end{equation}
\end{theorem}

\begin{proof}
According to the first-order optimality conditions, differentiating expression in bracket of (\ref{eqn:HJB}) with respect to $\phi$, the function $\phi^{*}(t)$ is given by
\begin{equation}
	\phi^{*}=-\frac{k\sigma\pi V_{a}}{V}. \label{phi}
\end{equation}
Substituting (\ref{phi}) into (\ref{eqn:HJB}), we have
\begin{equation}
	\begin{split}
		\mathop{\textrm{min}}_{\boldsymbol{\pi}\in \Pi}\big\{&V_t+\left[ar+NC(t)c{\rm e}^{\tau t}-NB(t)b{\rm e}^{\tau t}\right]V_a-(NC(t)\lambda_1+NB(t)\lambda_2)V_a\\
		&+\pi(\mu-r)V_a-\frac{k\sigma^{2}\pi^{2}V_a^{2}}{2V}+\frac{1}{2}\pi^2\sigma^2V_{aa}+\left(\gamma_1 \lambda_1^2+\gamma_2 \lambda_2^2\right)\rm{e}^{-rt} \big\}=0,
	\end{split}\label{min}
\end{equation}
Differentiating expression in bracket of (\ref{eqn:HJB}) with respect to $\pi$, $\lambda_1$ and $\lambda_2$, setting to zero and solving them give immediately
\begin{align}
\pi^{*}=\frac{\varphi V_{a}} {\sigma(\frac{kV_{a}^{2}}{V}-V_{aa})},\quad 
\lambda_1^{*}=\frac{NC(t)V_{a}}{2\gamma_1\rm{e}^{-rt}},\quad
\lambda_2^{*}=\frac{NB(t)V_{a}}{2\gamma_2\rm{e}^{-rt}}, \label{pi lambda1 lambda2}
\end{align}
Substituting (\ref{pi lambda1 lambda2}) into (\ref{min}), denote
$$g_{1}=\frac{NC(t)^{2}}{4\gamma_1{\rm e}^{-rt}}+\frac{NB(t)^{2}}{4\gamma_2{\rm e}^{-rt}},$$
$$g_{2}=NC(t)c{\rm e}^{\tau t}-NB(t)b{\rm e}^{\tau t},$$
we have
\begin{equation}\label{3.12}
V_t+\left[ar+g_{2}\right]V_a-g_{1}V_a^{2}+\frac{\varphi^{2}V_a^{2}}{2(\frac{kV_a^{2}}{V}-V_{aa})}=0,
\end{equation}
To solve (\ref{3.12}), we try to speculate on the solution in the following
form
\begin{equation}\label{eqn:boundary}
	\begin{split}
		V(t,a)=\gamma_3{\rm e}^{-rt}P(t)(a+Q(t))^{2},\\
		P(T)=1, \quad Q(T)=-a_{0}{\rm e}^{rT},
	\end{split}	
\end{equation}
whose partial derivatives are
\begin{equation}\label{eqn:differential}
	\begin{split}
	V_{t}&=\gamma_3{\rm e}^{-rt}\left[(P_{t}-rP(t))(a+Q(t))^{2}+2P(t)Q_{t}(a+Q(t))\right],\\
	V_{a}&=2\gamma_3{\rm e}^{-rt}P(t)(a+Q(t)),\\
	V_{aa}&=2\gamma_3{\rm e}^{-rt}P(t),
\end{split}	
\end{equation}

Substituting (\ref{eqn:differential}) into (\ref{3.12}), the optimal policies (control variables) can be expressed in terms of the functions $P(t)$, $Q(t)$.We have
\begin{equation}
	\begin{split}
		\gamma_3{\rm e}^{-rt}\big\{&(P_{t}+(r+\frac{\varphi^{2}}{2k-1})P(t)-4g_{1}\gamma_3{\rm e}^{-rt}P(t)^{2})a^{2}\\ 
		&+2((P_{t}+(\frac{\varphi^{2}}{2k-1}-r)P(t)-4g_{1}\gamma_3{\rm e}^{-rt}P(t)^{2})Q(t)+g_{2}P(t)+rP(t)Q(t)+P(t)Q_{t})a\\
		&+(P_{t}+(\frac{\varphi^{2}}{2k-1}-r)P(t)-4g_{1}\gamma_3{\rm e}^{-rt}P(t)^{2})Q(t)^{2}+2g_{2}P(t)Q(t)+2P(t)Q_{t}Q(t)\big\}=0,		
	\end{split}
\end{equation}

The coefficients of $a^2$, $a$ are all zeros, which leads to the following system of differential equations
\begin{equation}\label{PT,QT}
	\begin{split}
		&P_{t}+(r+\frac{\varphi^{2}}{2k-1})P(t)-4g_{1}\gamma_3{\rm e}^{-rt}P(t)^{2}=0,\\
		&2((P_{t}+(\frac{\varphi^{2}}{2k-1}-r)P(t)-4g_{1}\gamma_3{\rm e}^{-rt}P(t)^{2})Q(t)+g_{2}P(t)+rP(t)Q(t)+P(t)Q_{t})=0,	
	\end{split}
\end{equation}
Differential equations (\ref{PT,QT})with boundary conditions in (\ref{eqn:boundary}) can be easily solved. Denote $g_{3}=\frac{\varphi^{2}}{2k-1}$, we have
\begin{equation}
	\begin{split}	
		P(t)&=\frac{1}{R(t)},\\
		Q(t)&=\frac{g_{2}}{r}(1-{\rm e}^{-r(T-t)})-a_{0}{\rm e}^{rt},
	\end{split}
\end{equation}
where
\begin{equation}
	\begin{split}
		R(t)&={\rm e}^{-(r+g_{3})(T-t)}+\frac{4g_{1}\gamma_3}{2r+g_{3}}({\rm e}^{-rt}-{\rm e}^{-(2r+g_{3})T+(r+g_{3})t}).
	\end{split}
\end{equation}
\end{proof}

Above, we have derived closed-form expressions for the optimal
investment strategy and optimal risk-sharing arrangemnts taking into account
the three competing objectives of contibution rate stabilty and benefit stability and continous intergenerational risk-sharing. 

\begin{remark}
According to the closed-form result, it is not so obvious about the impact of the longevity parameter $\omega$. Actually, the $\pi^{*}$,$\lambda_1^{*}$,$\lambda_2^{*}$ is closely ralated to the $\omega$, since the expression of $NC(t)$ and $NB(t)$ contain survival function where $\omega$ as the important longevity parameter being shown in the force of mortality.
\end{remark}

Then, we will interpret the optimal controls presented in (\ref{pi}), (\ref{lambda_1}) and (\ref{lambda_2}), providing economic intuition for some of their more surprising behavior.

First, $Q(t)$ in (\ref{PT,QT}) can be rewrite in actuarial notation as
$$ Q(t)=g_2\overline{a}_{\overline{T-t|}r}-a_0{\rm e}^{rt},$$
where $\overline{a}_{\overline{T-t|}r}$ is the present value of an annuity certain payable continuously for the remainder of the horizon, valued at a force of interest $r$. Besides, $g_2$ in (\ref{g_1 g_2 g_3}) represent the difference between the target total contribution and target total benefit. Therefore the first term in $Q(t)$ is the present value at time $t$ of difference between the target total contribution and target total benefit during the remaining horizon, and the second is the accumulated asset at time $t$ for the initial fund. Now, we can have next remark.

\begin{remark}
The expression for $\pi^{*}$ in (\ref{pi}) can be rewritten as a constant proportion ($\frac{\varphi}{(2k-1)\sigma}$) of the excess or shortfall of its total assets and the present value of the plan's aspirations.
$$\pi^{*}=\frac{\varphi}{(2k-1)\sigma}(a+g_2\overline{a}_{\overline{T-t|}r}-a_0{\rm e}^{rt}),$$
Here total assets include both its financial assets $a(t)$ and the present value of defference between the target total contribution and target total benefit which will collect over the remainder of the horizon, whereas the plan's ``aspirations" reflects the target accumulated assets, with all items being valued at the risk-free rate, which is in line with \cite{Wang2018Optimal}. Consequently, investment in the risky asset is required only when the plan's total assets are not sufficient to cover the plan's aspirations by investing exclusively in the risk-free asset. 
\end{remark}

Next, moving our focus to the optimal risk-sharing adjustment policy. Like $Q(t)$, $P(t)$ can also be rewrite in actuarial notation as
\begin{align}
	\frac{1}{P(t)}=&{\rm e}^{-(\gamma-r)(T-t)}+4g_1\gamma_3{\rm e}^{-rt}\overline{a}_{\overline{T-t|}\gamma}\\
	=&{\rm e}^{-(\gamma-r)(T-t)}+(\frac{\gamma_3}{\gamma_1}NC(t)^{2}+\frac{\gamma_3}{\gamma_2}NB(t)^{2})\overline{a}_{\overline{T-t|}\gamma},
\end{align}
where $\gamma=2r+g_3$ and $\overline{a}_{\overline{T-t|}\gamma}$ is the present value of an annuity certain payable continuously for the remainder of the horizon, valued at a force of interest $\gamma$. Therefore 
$$\frac{\gamma_3}{\gamma_1}NC(t)^{2}P(t)=\frac{1}{\frac{\gamma_1}{\gamma_3NC(t)^{2}}{\rm e}^{-(\gamma-r)(T-t)}+(1+\frac{\gamma_1}{\gamma_2}(\frac{NB(t)}{NC(t)})^{2})\overline{a}_{\overline{T-t|}\gamma}},$$
where the denominator represent the present value of a continuous investment of $(1+\frac{\gamma_1}{\gamma_2}(\frac{NB(t)}{NC(t)})^{2})$ at time $t$ for $T-t$ years, valued at the adjusted rate $\gamma$, while leaving a lump sum at the end of $T-t$ years equal to $\frac{\gamma_1}{\gamma_3NC(t)^{2}}$, valued at the adjusted rate $\gamma-r$. And $\frac{\gamma_3}{\gamma_2}NB(t)^{2}P(t)$ can be showed as same way.

\begin{remark}
The expressions for $\lambda_1^{*}$ and $\lambda_2^{*}$ have similar pattern, and multiplying both side by $NC(t)$($NB(t)$) in \ref{lambda_1}(\ref{lambda_2}), we have
$$NC(t)\lambda_1^{*}=\frac{\gamma_3}{\gamma_1}NC(t)^{2}P(t)(a+Q(t)),$$
$$NB(t)\lambda_2^{*}=\frac{\gamma_3}{\gamma_2}NB(t)^{2}P(t)(a+Q(t)),$$
and we find that the optimal total contribution adjustments and optimal total benefit adjustments at time $t$ are dynamic adjustments that spreads the current excess (shortfall) of the plan's total assets relative to the plan's aspirations over the remaining number of years in the horizon with a proportionate share left over as part of the legacy fund at the end of the horizon.

\end{remark}

\section{Special Cases}
\subsection{Optimal investment problem for a hybrid pension without model uncertainty}
Under the assumptions of no uncertainty, as we mentioned in section 3.1, when $\Psi=0$ in (\ref{J(t,a,w)}) , the manager is completely convinced that the true measure is the reference measure $\mathbb{P}$. In this case, problem (\ref{J(t,a,w)}) degenerates to the classical quadratic deviations minization in the absence of ambiguity. since $\Psi(t)=k/V$, where $k$ is positive constants and stand for the ambiguity aversion parameter, meaning manager's the different levels of ambiguity aversion about the stock’s price, therefore, we can set $k=0$ in problem (\ref{J(t,a,w)}). The following is the classical quadratic deviations minization problem in absence of model uncertainty.              
\begin{equation} 
	\begin{split}
		V(t,a)&=\mathop{\textrm{min}}_{\boldsymbol{\pi}\in \Pi} \bigg\{\int_{t}^{T} {\left[\gamma_1(c(u)-c{\rm e}^{\tau t})^2{\rm e}^{-ru}+\gamma_2(b(u)-b\rm{e}^{\tau t})^2{\rm e}^{-ru}\right]} \textrm{d}u\\
		&\quad \quad \quad \quad \quad \quad \quad \quad +\gamma_3(A(T)-a_0{\rm e}^{rT})^2{\rm e}^{-rT}  \bigg\} 
	\end{split}\label{J(t,a,w)4.1}
\end{equation}

The result of problem (\ref{J(t,a,w)4.1}) is given by following Theorem (\ref{thm:main4.1}), which is derived by set $k=0$ in Theorem \ref{thm:main}.
\begin{theorem} \label{thm:main4.1} \rm  the optimal asset allocation policy, contribution and benefit adjustment policy are given, respectively, by		
	\begin{align}
		\pi^*(t,a)&=-\frac{\varphi}{\sigma}(a+Q(t)), \label{pi4.1} \\
		\lambda_1^*(t,a)&=\frac{\gamma_3}{\gamma_1}NC(t)P(t)(a+Q(t) ), \label{lambda_14.1}  \\
		\lambda_2^*(t,a)&=\frac{\gamma_3}{\gamma_2}NB(t)P(t)(a+Q(t) ),   \label{lambda_24.1}
	\end{align}
	where $NC(t)$ and $NB(t)$ is show as (\ref{NC}) and (\ref{NB}) respectively. and the corresponding value function is given by
	\begin{equation}
		V(t,a)=\gamma_3{\rm e}^{-rt}P(t)(a+Q(t))^{2},
	\end{equation} 
	where the expressions of $P(t)$ and $Q(t)$ are given below. We denote that
	
	\begin{equation}\label{g_1 g_2 g_3,4.1}
		g_{1}=\frac{NC(t)^{2}}{4\gamma_1\rm{e}^{-rt}}+\frac{NB(t)^{2}}{4\gamma_2\rm{e}^{-rt}};\quad \quad 
		g_{2}=NC(t)c{\rm e}^{\tau t}-NB(t)b{\rm e}^{\tau t};\quad \quad
		g_{3}=-\varphi^{2},
	\end{equation} 
	
	then $P(t)$ and $Q(t)$ can be show as 
	\begin{equation}
		\begin{split}	
			P(t)&=\frac{1}{R(t)},\\
			Q(t)&=\frac{g_{2}}{r}(1-{\rm e}^{-r(T-t)})-a_{0}{\rm e}^{rt},
		\end{split}
	\end{equation}
	where
	\begin{equation}
		\begin{split}
			R(t)&={\rm e}^{-(r+g_{3})(T-t)}+\frac{4g_{1}\gamma_3}{2r+g_{3}}({\rm e}^{-rt}-{\rm e}^{-(2r+g_{3})T+(r+g_{3})t}).
		\end{split}
	\end{equation}
\end{theorem}

\begin{remark}
$g_1$ and $g_2$ are unaffected by uncertainty, and $g_3$ is associated with uncertainty.  It can be seen that the proportion of investment in risky assets decreases with the increase of time and is lower than when there are exist model uncertainty. The contribution is lower than that with model uncertainty, and the retirement benefits are higher than that with model uncertainty. This is mainly because pension managers have confidence in the model and do not have the risk of model uncertainty.  For detailed analysis, see numerical analysis in section 5.  
\end{remark}

\subsection{Optimal investment problem for a hybrid pension without consideration of longevity trend}
Under the assumptions of on longevity trend, we can set $\omega=\infty$ in Assumption 2.1 and $\xi=0$ in Assumption 2.2. Therefore the maximum survival age is $m_0$ which is constant and the survival function $p(x)$ is denoted by
\begin{equation}
p(x)={\rm e}^{-\mathcal{A}(x-x_0)-\frac{\mathcal{B}}{ln\theta} (\theta^{x}-\theta^{x_0})}, \qquad m_0\geq x \ge x_0.\label{p(x)}
\end{equation}
which is as same as the \cite{2019Optimal}. Besides, the optimal investment problem is as same as the (\ref{J(t,a,w)}) which is given by 
\begin{equation} 
	\begin{split}
		V(t,a)&=\mathop{\textrm{min}}_{\boldsymbol{\pi}\in \Pi,\phi}\mathop{\textrm{max}}_{\mathbb{Q}^{\phi} \in \mathscr{Q}} \textrm{E}_{(\boldsymbol{\pi},\phi)}^{\mathbb{Q}} \bigg\{\int_{t}^{T} {\left[\gamma_1(c(u)-c{\rm e}^{\tau t})^2{\rm e}^{-ru}+\gamma_2(b(u)-b{\rm e}^{\tau t})^2{\rm e}^{-ru}\right]} \textrm{d}u\\
		&\quad \quad \quad \quad \quad \quad \quad \quad +\gamma_3(A(T)-a_0{\rm e}^{rT})^2{\rm e}^{-rT}+	\int_{t}^{T}{\frac{(\phi(u))^2}{2\Psi(u)}}   \textrm{d}u \bigg\},  
	\end{split}
\end{equation}

The result is given by following Theorem \ref{thm:main4.2}, which is derived by set $\omega=\infty$ and $\xi=0$ in Theorem \ref{thm:main}.
\begin{theorem} \label{thm:main4.2} \rm The optimal asset allocation policy, contribution and benefit adjustment policy are given, respectively, by		
	\begin{align}
		\phi^{*}&=-\frac{k\varphi}{2(2k-1)},\\
		\pi^*(t,a)&=\frac{\varphi}{(2k-1)\sigma}(a+Q(t)), \label{pi4.2} \\
		\lambda_1^*(t,a)&=\frac{\gamma_3}{\gamma_1}NC(t)P(t)(a+Q(t) ), \label{lambda_14.2}  \\
		\lambda_2^*(t,a)&=\frac{\gamma_3}{\gamma_2}NB(t)P(t)(a+Q(t) ),   \label{lambda_24.2}
	\end{align}
	where $NC(t)$ and $NB(t)$ is show as
	 \begin{equation}
		NC(t)=\int_{x_0}^{x_r} {n(t-(x-x_0))p(x)} \textrm{d}x, \label{NC4.2}
	\end{equation}
	and the total number of retired members who recieve the benefits at time $t$ is
	\begin{equation}
		NB(t)=\int_{x_r}^{m_0} {n(t-(x-x_0))p(x)} \textrm{d}x. \label{NB4.2}
	\end{equation}
where the $p(x)$ is given by (\ref{p(x)}).

 The corresponding value function is given by
	\begin{equation}
		V(t,a)=\gamma_3{\rm e}^{-rt}P(t)(a+Q(t))^{2},
	\end{equation} 
	where the expressions of $P(t)$ and $Q(t)$ are given below. We denote that
	
	\begin{equation}\label{g_1 g_2 g_3.4.2}
		g_{1}=\frac{NC(t)^{2}}{4\gamma_1{\rm e}^{-rt}}+\frac{NB(t)^{2}}{4\gamma_2{\rm e}^{-rt}};\quad \quad 
		g_{2}=NC(t)c{\rm e}^{\tau t}-NB(t)b{\rm e}^{\tau t};\quad \quad
		g_{3}=\frac{\varphi^{2}}{2k-1},
	\end{equation} 
	
	then $P(t)$ and $Q(t)$ can be show as 
	\begin{equation}
		\begin{split}	
			P(t)&=\frac{1}{R(t)},\\
			Q(t)&=\frac{g_{2}}{r}(1-{\rm e}^{-r(T-t)})-a_{0}{\rm e}^{rt},
		\end{split}
	\end{equation}
	where
	\begin{equation}
		\begin{split}
			R(t)&={\rm e}^{-(r+g_{3})(T-t)}+\frac{4g_{1}\gamma_3}{2r+g_{3}}({\rm e}^{-rt}-{\rm e}^{-(2r+g_{3})T+(r+g_{3})t}).
		\end{split}
	\end{equation}
\end{theorem}

\begin{remark}
$g_1$ and $g_2$ were affected by the longevity trend, while $g_3$ was not. It can be seen that the proportion of investment in risky assets is lower than that with longevity trend, the optimal contribution is lower than that with longevity trend, and the optimal retirement benefits are higher than that with longevity trend. The main reason is that without the influence of longevity trend, the pension system has no aging problem, while the actual pension system has the influence of longevity trend.  Therefore, pension systems with longevity trends must be considered. For detailed analysis, see numerical analysis in section 5.  
\end{remark}

\subsection{Optimal investment problem for a hybrid pension without model uncertainty and consideration of longevity}

Under the assumptions of without uncertainty and longevity trend, we set $k=0$, $\omega=\infty$, $\xi=0$. The optimal investment problem is as same as (\ref{J(t,a,w)4.1}) and is given by
\begin{equation} 
	\begin{split}
		V(t,a)&=\mathop{\textrm{min}}_{\boldsymbol{\pi}\in \Pi} \bigg\{\int_{t}^{T} {\left[\gamma_1(c(u)-c{\rm e}^{\tau t})^2{\rm e}^{-ru}+\gamma_2(b(u)-b\rm{e}^{\tau t})^2{\rm e}^{-ru}\right]} \textrm{d}u\\
		&\quad \quad \quad \quad \quad \quad \quad \quad +\gamma_3(A(T)-a_0\rm{e}^{rT})^2\rm{e}^{-rT}  \bigg\} 
	\end{split}
\end{equation}

The result is given by following Theorem \ref{thm:main4.3},  which is derived by set $k=0$,$\omega=\infty$ and $\xi=0$ in Theorem \ref{thm:main}.
\begin{theorem} \label{thm:main4.3} \rm  the optimal asset allocation policy, contribution and benefit adjustment policy are given, respectively, by
	\begin{align}
		\pi^*(t,a)&=-\frac{\varphi}{\sigma}(a+Q(t)),  \\
		\lambda_1^*(t,a)&=\frac{\gamma_3}{\gamma_1}NC(t)P(t)(a+Q(t) ),   \\
		\lambda_2^*(t,a)&=\frac{\gamma_3}{\gamma_2}NB(t)P(t)(a+Q(t) ),   
	\end{align}
	where $NC(t)$ and $NB(t)$ is show as (\ref{NC4.2}) and (\ref{NB4.2}) respectively. and the corresponding value function is given by
	\begin{equation}
		V(t,a)=\gamma_3{\rm e}^{-rt}P(t)(a+Q(t))^{2},
	\end{equation} 
	where the expressions of $P(t)$ and $Q(t)$ are given below. We denote that
	
	\begin{equation}\label{g_1 g_2 g_3,4.3}
		g_{1}=\frac{NC(t)^{2}}{4\gamma_1{\rm e}^{-rt}}+\frac{NB(t)^{2}}{4\gamma_2{\rm e}^{-rt}};\quad \quad 
		g_{2}=NC(t)c{\rm e}^{\tau t}-NB(t)b{\rm e}^{\tau t};\quad \quad
		g_{3}=-\varphi^{2},
	\end{equation} 
	
	then $P(t)$ and $Q(t)$ can be show as 
	\begin{equation}
		\begin{split}	
			P(t)&=\frac{1}{R(t)},\\
			Q(t)&=\frac{g_{2}}{r}(1-{\rm e}^{-r(T-t)})-a_{0}{\rm e}^{rt},
		\end{split}
	\end{equation}
	where
	\begin{equation}
		\begin{split}
			R(t)&={\rm e}^{-(r+g_{3})(T-t)}+\frac{4g_{1}\gamma_3}{2r+g_{3}}({\rm e}^{-rt}-{\rm e}^{-(2r+g_{3})T+(r+g_{3})t}).
		\end{split}
	\end{equation}
\end{theorem}

\begin{remark}
Regardless of model uncertainty and longevity trends, this study is very similar to the work of \cite{2020Optimal} and \cite{2019Optimal}, and the results are similar.  Therefore, the results of this paper include the research results of \cite{2020Optimal} and \cite{2019Optimal}.
\end{remark}

\section{Numerical illustrations}
In section 3, we have establish the explicit expression of the robust optimal investiment allocation and adjustment strategies of both the contribution and the benefit for per cohort in the hybrid pension model. In this section, we present numerical examples to illustrate our results regarding optimal policies.

The parameter settings in numercial simulation are mainly in line with the ones in literature \cite{2019Optimal}. In the demography aspect, we assume a negative growth rate of new labor entrants to study the impacts of aging problem ($\kappa=-0.005$), which depicts the scenario of negative growth of labor supply. The assumption is far removed from current reality but might emerge soon. For the age and time-dependent survival function, set $\theta=1.124$, $\mathcal{A}=0.000022$, $\mathcal{B}=2.7 \time 10^{-6}$. Besides, we set $\omega=4$, which means that every 4 years, the average life expectancy increases by one year.
\begin{table}[H]
	\centering
	\caption{parameter settings in the numercial simulation}
	\label{Table 1}
	\begin{tabular}{ccccccccccc}		
		\toprule
		$\kappa$ &  $\theta$  &$\mathcal{A}$&$\mathcal{B}$&$\omega$ & $x_0$ & $x_r$& $n_0$& $a_0$&$T$&$m_0$\\
		-0.005& 1.124 &0.000022& $2.7\times 10^{-6}$ &4&25&65&10&3000&20&100\\
		\midrule
		 $c$& $b$&$\tau$& $\gamma_1$& $\gamma_2$&$\gamma_3$ & $r$ & $\mu$ & $\sigma$ & $k$&$\xi$\\
		0.1&0.7&0.02&2&2&2&0.01&0.05&0.15 & 2&0.25\\  
		\bottomrule	 
	\end{tabular}
\end{table}	
In the pension rules aspect, the age of the new entrants into the pension system is $x_0=25$, and the age of retirement is $x_r=65$. The number of new entrants at time 0 is $ n_0=10$. The initial fund accumulation is $a_0=3000$. The time horizon is $T=20$. Futhermore, the target contribution is $c=0.1$, and target benefit is $b=0.7$.  And the target growth at detemnistic rate $\tau=0.02$. We assume that the weight parameter $\gamma_1=2,\gamma_2=2,\gamma_3=2$, and we will change the vaule and analyse the impact.

In the finacial market aspect, the risk-free interest rate is $r=0.01$. The expected return and the volatility of the risky asset is $\mu=0.05$ and $\sigma=0.15$, respectively. As such, the sharpe ratio is $\varphi=0.2667$. At last, the ambiguity aversion parameter $k=2$.

In the following part, we use Monte Carlo method to do the numerical analysis. We first partition the time interval $[0,20]$ into 1000 subintervals, For each of these subintervals, we generate a random variable with normal distribution, which applied to the common random source of risky asset. Based on the salary level and risky asset price, as well as the optimal control policies obtained in section 3, we could get the new fund accumulation at the end of the subinterval. Then, we recalculate the optimal contribution, the optimal benefit and the optimal asset allocation amount based on the new accumulation, using (\ref{c(t)}), (\ref{b(t)}), (\ref{pi}), (\ref{lambda_1}), (\ref{lambda_2}), following the above instructions, the procedure is repeated in the next subinterval using the newly determined value of $\pi^{*},c^{*},b^{*}$, yielding 1000 consecutive values of $\pi^{*},c^{*},b^{*}$ during the time interval $[0,20]$. By repeating this process 1000 times, we obtain the empirical distribution of $\pi^{*},c^{*},b^{*}$ at the end of each subinterval, and the average of the optimal contribution rate $\mathbb{E}c^{*}(t)$, the average of the optimal benefit $\mathbb{E}b^{*}(t)$ and the average of the optimal risky investment proportion $\mathbb{E}(\frac{\pi^{*}(t)}{A^{*}(t)})$ are used by the follow-up analysis.

According to the parameter settings in the Table \ref{Table 1}. Figure \ref{alpha&beta} gives the information about the spread parameter for rcontribution income and benefit payments $\alpha$, $\beta$ in the hybrid pension system. Both the $\alpha$ and $\beta$ have similar trend. As we mentioned in section 2, the lower vaule of $\alpha$ and $\beta$, the higher degree of intergenerational risk sharing. As such, the pension system gradually decreases the part of funding surplus which shifted to the future, i.e, decreases the degree of intergenerational risk sharing.
\begin{figure} 
	\centering
	\subfigure[]{\includegraphics[scale=0.3]{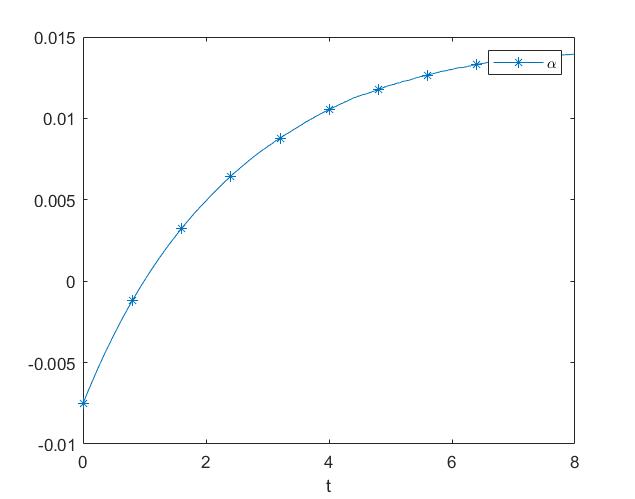}}
    \subfigure[]{\includegraphics[scale=0.3]{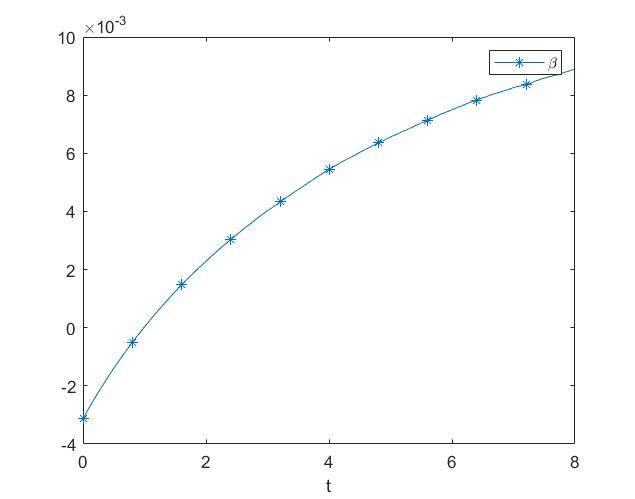}}
	\caption{The spread parameter for contribution income and benefit payment $\alpha$,  $\beta$ according to time}
	\label{alpha&beta} 
\end{figure}

In the first part of sensitivity analysis, we study the impacts of demography trends and time interval on the optimal control policies with respect to time $t$.

First, we investigate the impacts of new entrants growth rate $\kappa$ on the optimal polices for the hybrid pension model. In Figure \ref{c&b_kappa}, when the new entrants growth rate is decreases, the dependency ratio increases and it leads to higher deficit of the pension budget. As such, the contribution substantially increases and benefit decrease at the same time. In Figure \ref{pi_kappa&k}, when the new entrants growth rate decreases, the proportion allocated to the risky asset increases. The pension manager needs to take more investment risk to re-balance the pension budget. Besides, we notice that when $\kappa=0$ which means that the number of population entering the pension plan system is $n_0$ every year and the population stays relatively stable during this 20 years. In that case, In the first two years, the contribution and optimal investment proportion are negative since there are no much pressure on finacial at the beginning.


\begin{figure} 
	\centering
	\subfigure[]{\includegraphics[scale=0.3]{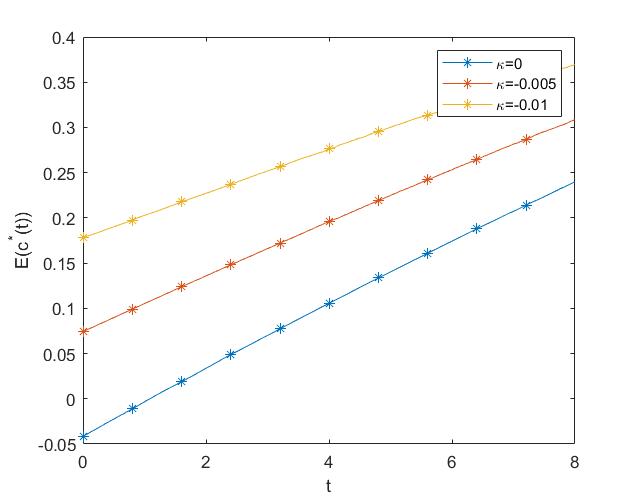}}
	\subfigure[]{\includegraphics[scale=0.3]{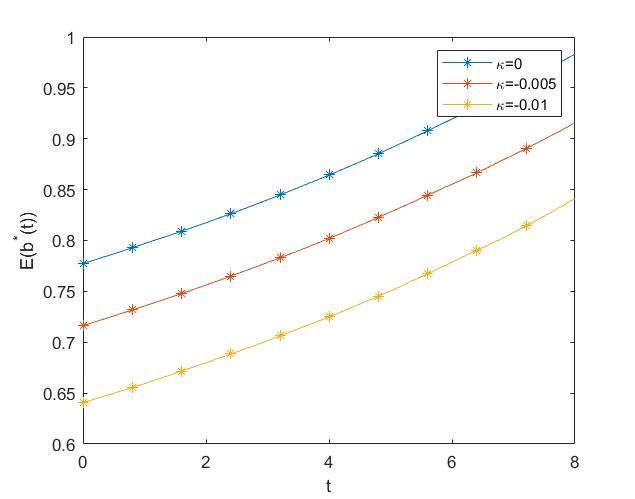}}
	\caption{The impacts of new entrants growth rate $\kappa$ on optimal contribution $\mathbb{E}c^{*}(t)$ and optimal benefit $\mathbb{E}b^{*}(t)$}  \label{c&b_kappa}
\end{figure}
\begin{figure}[H] 
    \centering
    \includegraphics[scale=0.3]{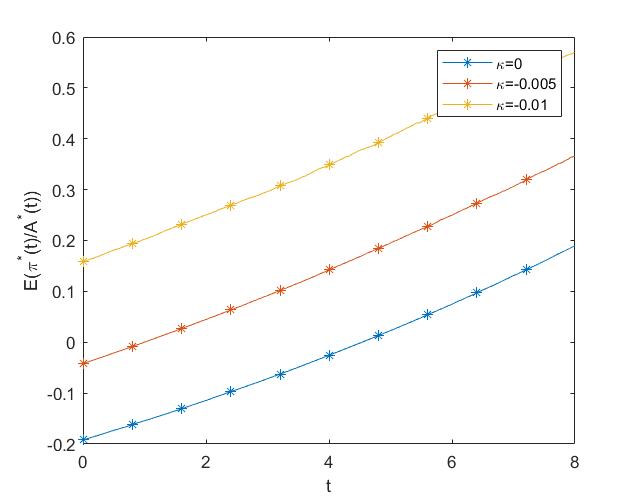}
	\caption{The impacts of new entrants growth rate $\kappa$  on optimal investment proportion $\mathbb{E}(\frac{\pi^{*}(t)}{A^{*}(t)})$ } \label{pi_kappa&k}
\end{figure}


\begin{figure} [H]
	\centering
	\subfigure[]{\includegraphics[scale=0.3]{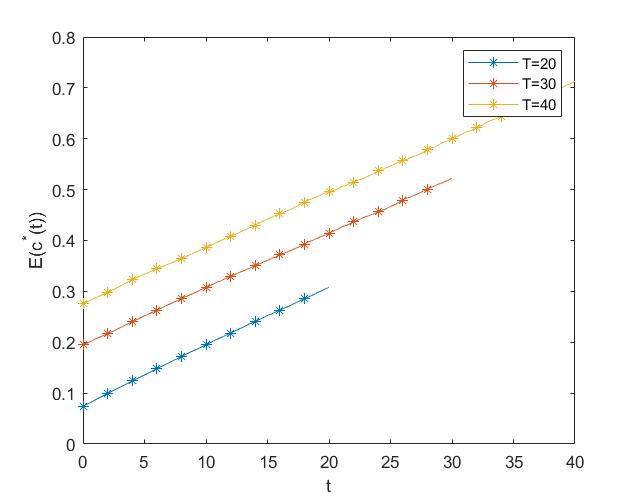}}
	\subfigure[]{\includegraphics[scale=0.3]{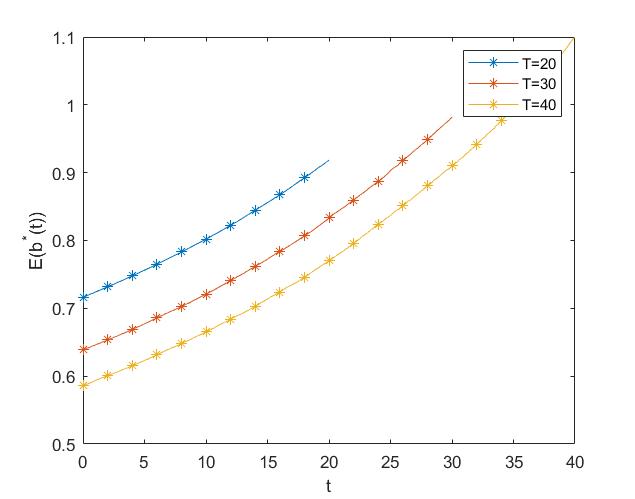}}
	\caption{The impacts of time interval $T$ on optimal contribution $\mathbb{E}c^{*}(t)$ and optimal benefit $\mathbb{E}b^{*}(t)$}  \label{c&b_T}
\end{figure}

Then, we investigate the impact of time interval $T$ on optimal policies. In Figure \ref{c&b_T}, long time interval leads to higher contribution and lower benefit since there are more uncertainty in the remote future and we have to put sufficent money into pension system to meet the budget. In Figure \ref{pi_T&xr_omega} (a), when time interval we considered is relatively long, there are more proportion put into risky asset since pension system needs more money to keep sustainable. 
\begin{figure}[H] 
 \centering
 \includegraphics[scale=0.3]{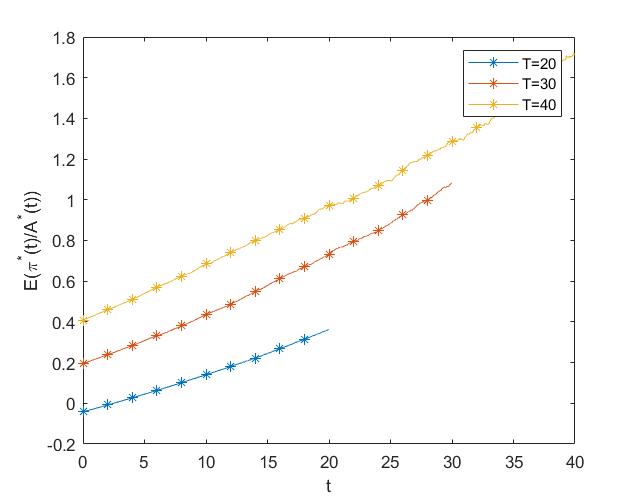}
	
	\caption{The impacts of time interval $T$ on optimal investment proportion $\mathbb{E}(\frac{\pi^{*}(t)}{A^{*}(t)})$ } \label{pi_T&xr_omega}
\end{figure}

In the second part of sensitivity analysis, we study the cross impacts of two dimensional parameters on the optimal policies at time 10.

First, we investigate the impacts of retirement age $x_r$ and longevity parameter $\omega$ on optimal policies. In Figure \ref{c&b_xr_omega} and Figure \ref{pi_xr_kappa} (a). Lower $\omega$ means the aging problem is relatively urgent, which leads to higher contribution, higher risky asset proportion and lower benefit. That is because there will be more retired people in our system, higher contribution and risky asset proportion can increase the asset so that keep our pension system stable and sustainable. At the same time, if we can postponing retirement age, the risky asset proportion and contribution will slightly decrease and the benefit will increase, since there are more working cohort in pension system and sufficient accumulations as the benefit payment. This can alleviate the contribution burden and reduces both of the risk-taking and discontinutity risk in some aspects. Therefore, according to the change of longevity parameter $\omega$, we can choose the appropriate retirement age to make sure the pension system run smoothly.

\begin{figure}[H] 
	\subfigure[]{\includegraphics[scale=0.3]{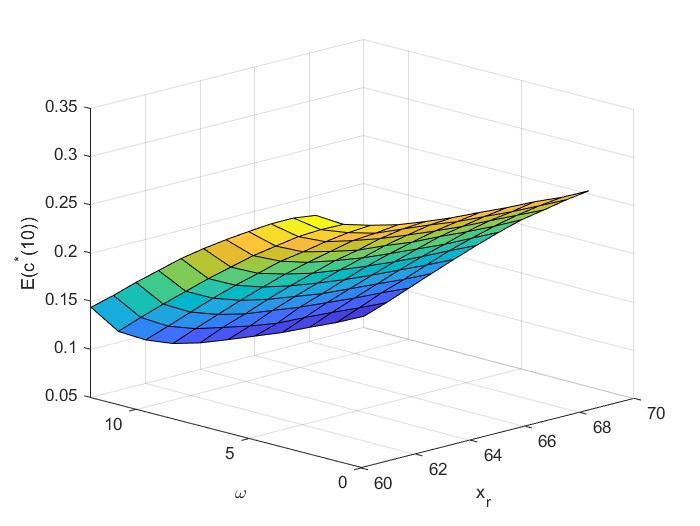}}
	\subfigure[]{\includegraphics[scale=0.3]{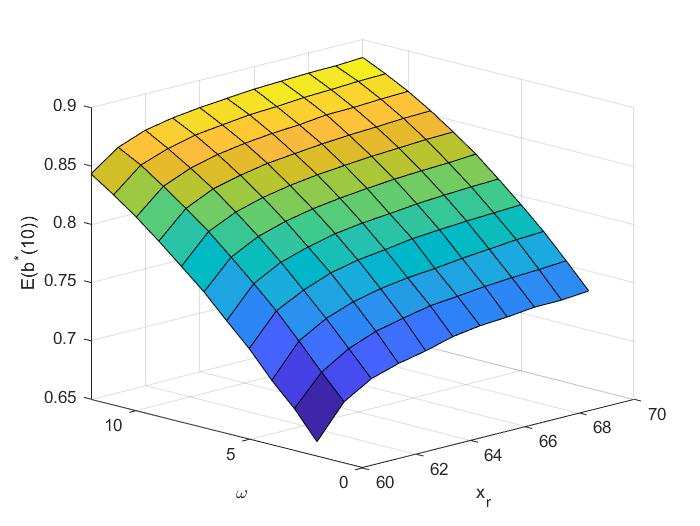}}
	\caption{The impacts of retirement age $x_r$ and $\omega$  on optimal contribution  $\mathbb{E}c^{*}(t)$ and optimal benefit $\mathbb{E}b^{*}(t)$}  \label{c&b_xr_omega}
\end{figure}
\begin{figure}[H] 
	\centering 
	\subfigure[]{\includegraphics[scale=0.3]{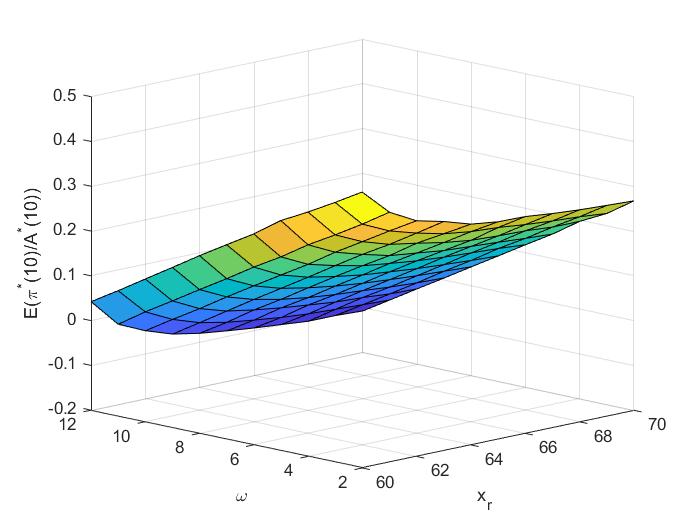}}	
	\subfigure[]{\includegraphics[scale=0.3]{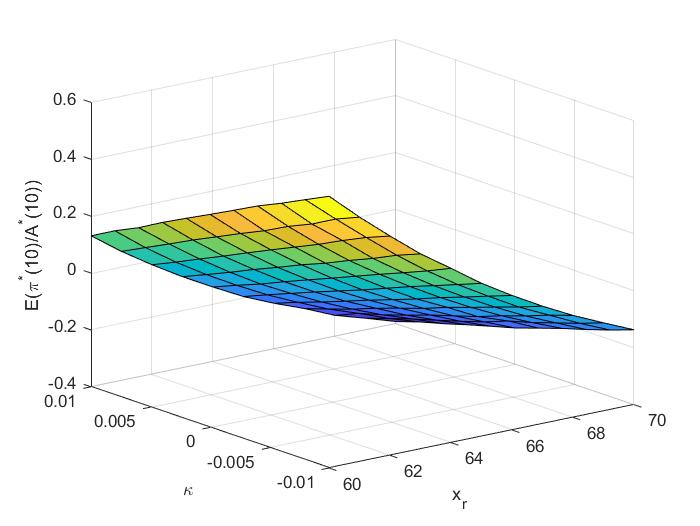}}
	\caption{The impacts of retirement age $x_r$ and $\omega$, retirement age $x_r$ and new entrants growth rate $\kappa$ on optimal investment proportion $\mathbb{E}(\frac{\pi^{*}(t)}{A^{*}(t)})$}
	\label{pi_xr_kappa} 
\end{figure}

Second, we investigate the impacts of retirement age $x_r$ and new entrants growth rate $\kappa$ on optimal policies. In Figure \ref{pi_xr_kappa}(b) and Figure \ref{c&b_xr_kappa}, when the new entrants growth rate is low and people retire early, the contribution rate and the risky investment proportion are both high. Meanwhile, we find that postponing the retirement age from 60 to 70 could offset the negative impact of decreasing new entrants growth rate from 0.01 to -0.01. Therefore, postponing the retirement is an important measure to alleviate the adverse influence of the aging problem. This conlusion is in line with the \cite{2020Optimal}.

\begin{figure} 
	\centering
	\subfigure[]{\includegraphics[scale=0.3]{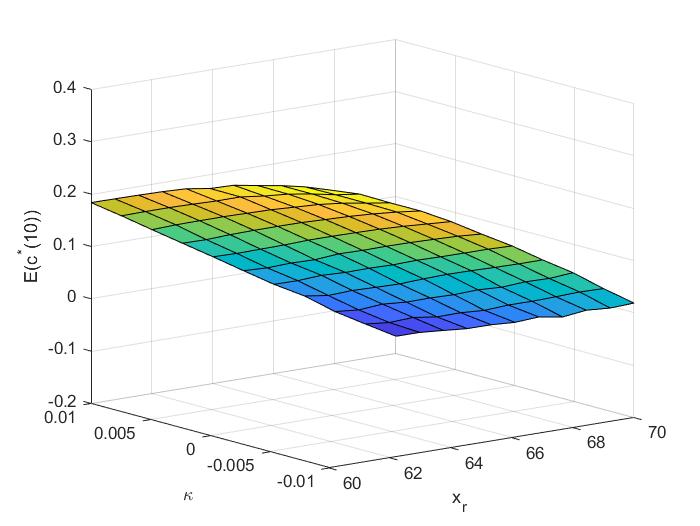}}
	\subfigure[]{\includegraphics[scale=0.3]{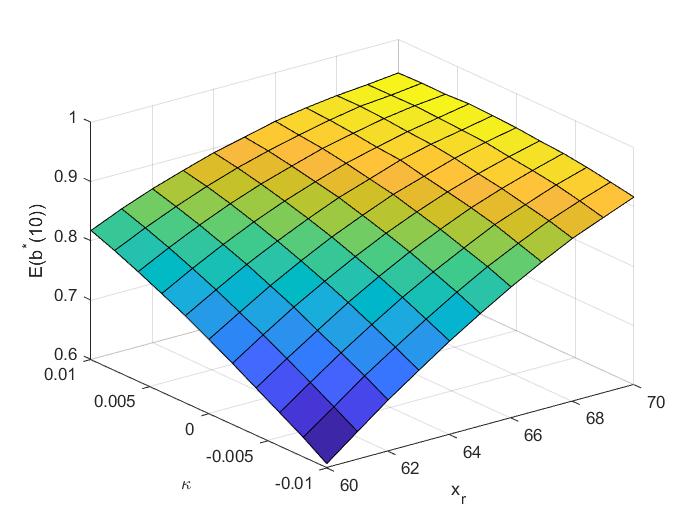}}
	\caption{The impacts of retirement age $x_r$ and new entrants growth rate $\kappa$  on optimal contribution  $\mathbb{E}c^{*}(t)$ and optimal benefit $\mathbb{E}b^{*}(t)$}  \label{c&b_xr_kappa}
\end{figure}

At last, we investigate the impacts of the weight parameter of unstable contribution $\gamma_1$ and benefit risk $\gamma_2$ on optimal policies. In Figure \ref{c&b_gamma1_gamma2}, when $\gamma_1$ is high and $\gamma_2$ is low, which means we are more care about the unstable contribution risk instead of unstable benefit risk, both the contribution and benefit are high. This shows that as long as keeping the contribution stable and adequate, the benefit is natually can be statisfied.

\begin{figure} [H]
	\centering
	\subfigure[]{\includegraphics[scale=0.25]{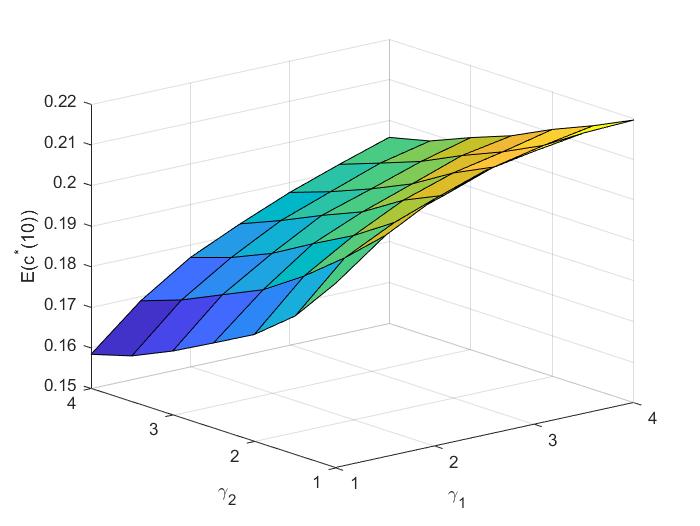}}
	\subfigure[]{\includegraphics[scale=0.25]{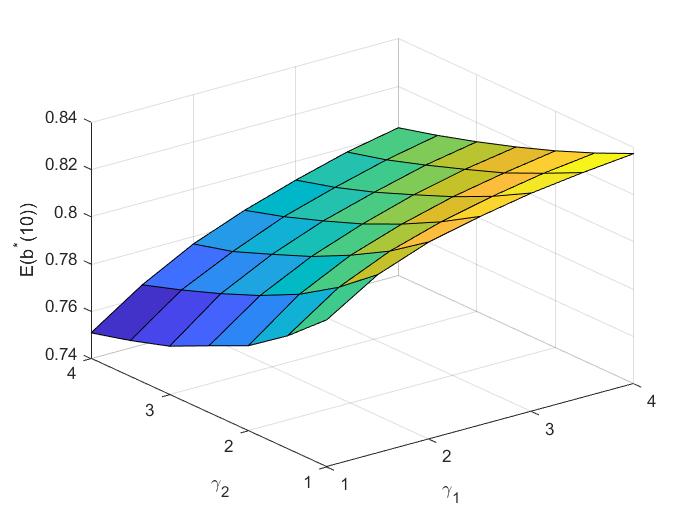}}
	\caption{The impacts of $\gamma_1$ and $\gamma_2$ on optimal contribution $\mathbb{E}c^{*}(t)$ and optimal benefit $\mathbb{E}b^{*}(t)$}  \label{c&b_gamma1_gamma2}
\end{figure}

In the third part of sensitivity analysis, we study the special cases that we discussed in Section 4. 

First, under the assumptions of no uncertainty, Figure \ref{c&b_k} and \ref{pi_k} is the numercial analysis compared with the $k=2$ and $k=4$. When there is no model uncertainty, the contrubtion is relatively low and the benefit is raltively high. Besides, the investment proportion has a total different pattern. That is because  the pension manager has faith to the true model and no risk of model uncertainty. 
\begin{figure} [H]
	\centering
	\subfigure[]{\includegraphics[scale=0.28]{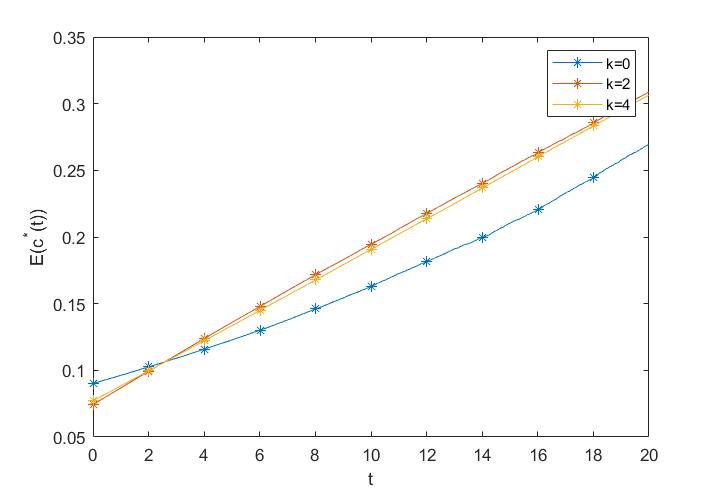}}
	\subfigure[]{\includegraphics[scale=0.28]{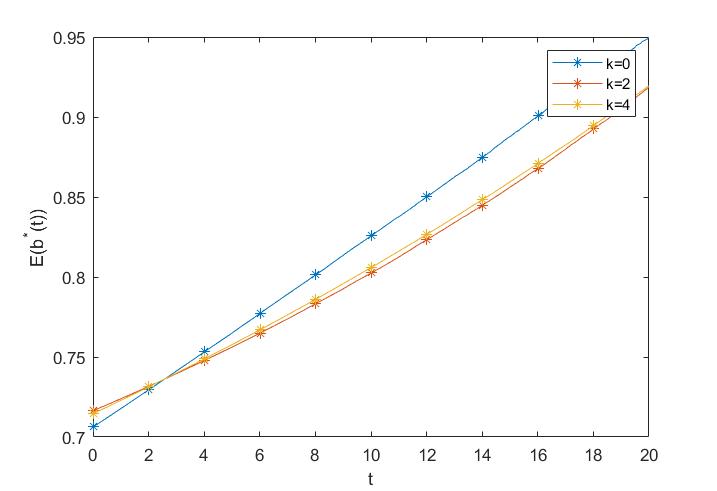}}
	\caption{The impacts of uncertainty on optimal contribution $\mathbb{E}c^{*}(t)$ and optimal benefit $\mathbb{E}b^{*}(t)$}  \label{c&b_k}
\end{figure}

\begin{figure}[H] 
	\centering
	\includegraphics[scale=0.3]{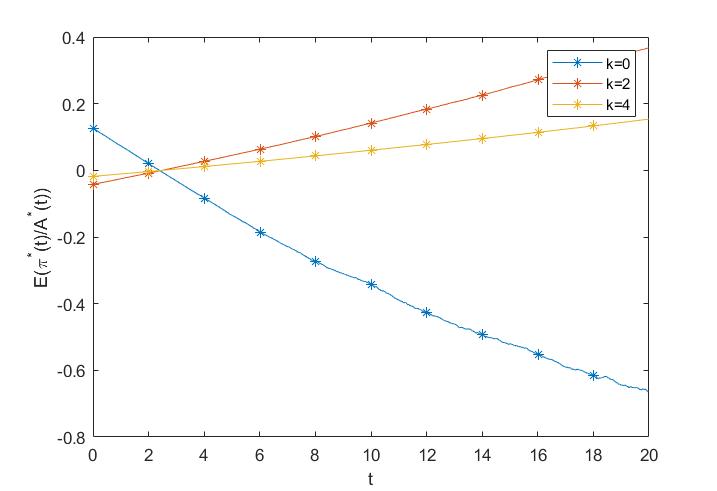}
	
	\caption{The impacts of  uncertainty on optimal investment proportion $\mathbb{E}(\frac{\pi^{*}(t)}{A^{*}(t)})$ } \label{pi_k}
\end{figure}

Then, under the assumptions of on longevity trend, Figure \ref{c&b_longevity} and \ref{pi_longevity} is the numercial analysis compared with the model with longevity trend. When there is no longevity trend, the contrubtion and investment proportion is relatively low and the benefit is raltively high. Without longevity trend, there are not too much pressures on the pension plan, therefore, it allows that a few proportion of assets invested to the finicial market. 

\begin{figure} [H]
	\centering
	\subfigure[]{\includegraphics[scale=0.3]{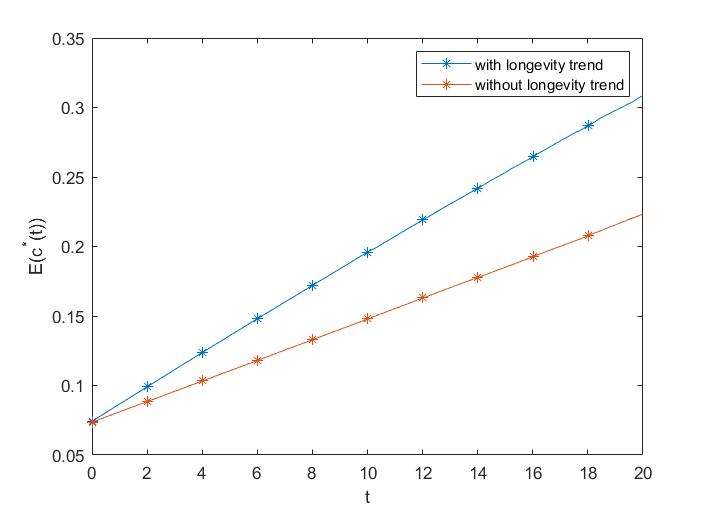}}
	\subfigure[]{\includegraphics[scale=0.3]{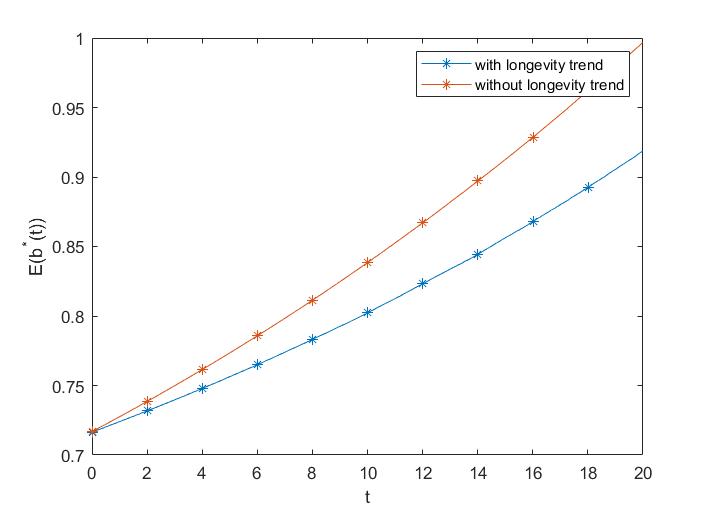}}
	\caption{The impacts of longevity trend on optimal contribution $\mathbb{E}c^{*}(t)$ and optimal benefit $\mathbb{E}b^{*}(t)$}  \label{c&b_longevity}
\end{figure}

\begin{figure}[H] 
	\centering
	\includegraphics[scale=0.3]{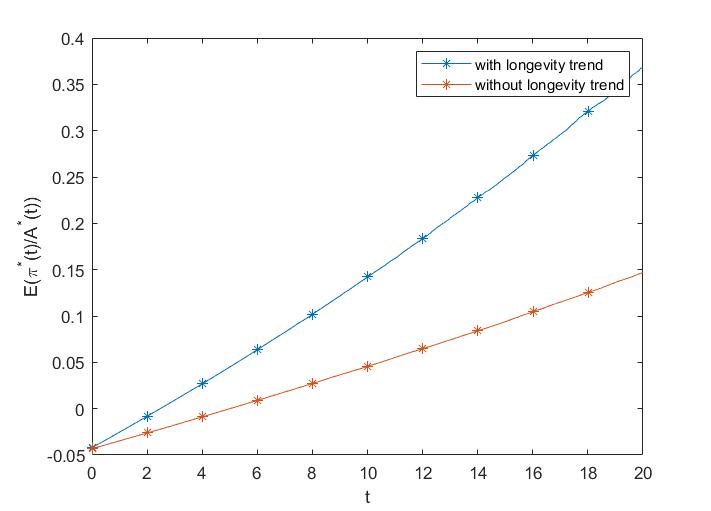}
	
	\caption{The impacts of longevity trend on optimal investment proportion $\mathbb{E}(\frac{\pi^{*}(t)}{A^{*}(t)})$ } \label{pi_longevity}
\end{figure}

At last, under the assumptions of without uncertainty and longevity trend, the contrubtion is relatively low and the benefit is raltively high. Besides, the investment proportion has a total different pattern. Without uncertainty and longevity trend, pension manager has faith to the true model and no pressure of aging problem. Thus, it is reasonable to decrease the proportion of investment to risk asset and the contributions, and retirees can get more pension after they retired which is in line with the \cite{2020Optimal} and \cite{Wang2018Optimal}.

\begin{figure} [H]
	\centering
	\subfigure[]{\includegraphics[scale=0.3]{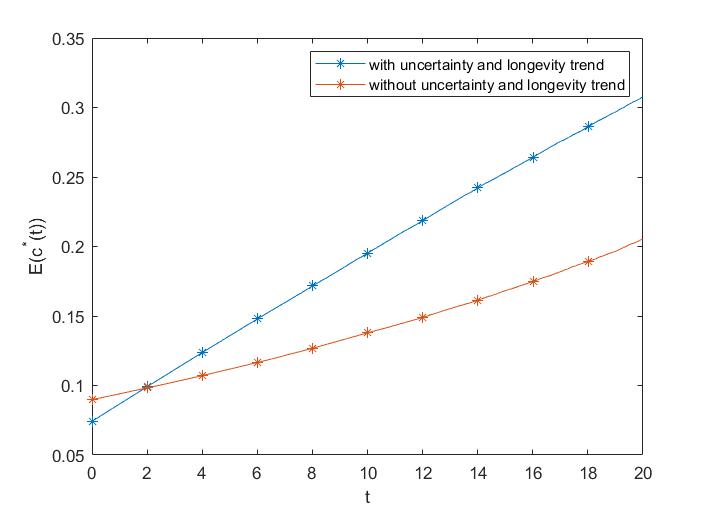}}
	\subfigure[]{\includegraphics[scale=0.3]{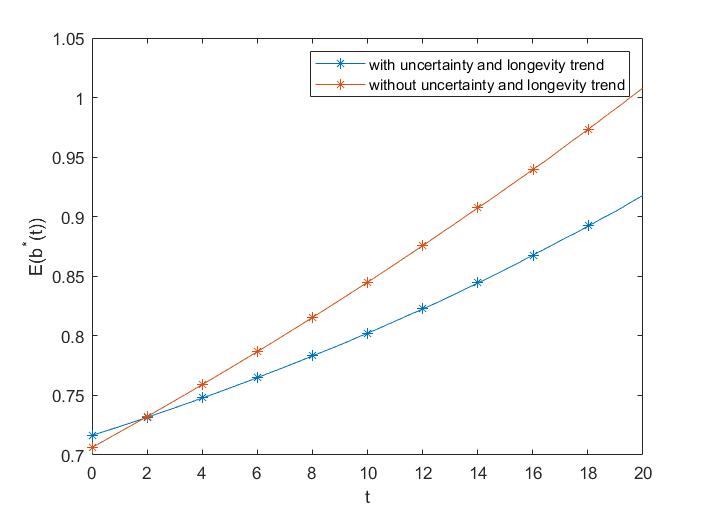}}
	\caption{The impacts of uncertainty and longevity trend on optimal contribution $\mathbb{E}c^{*}(t)$ and optimal benefit $\mathbb{E}b^{*}(t)$}  \label{c&b_longevity_uncertainty}
\end{figure}

\begin{figure}[H] 
	\centering
	\includegraphics[scale=0.3]{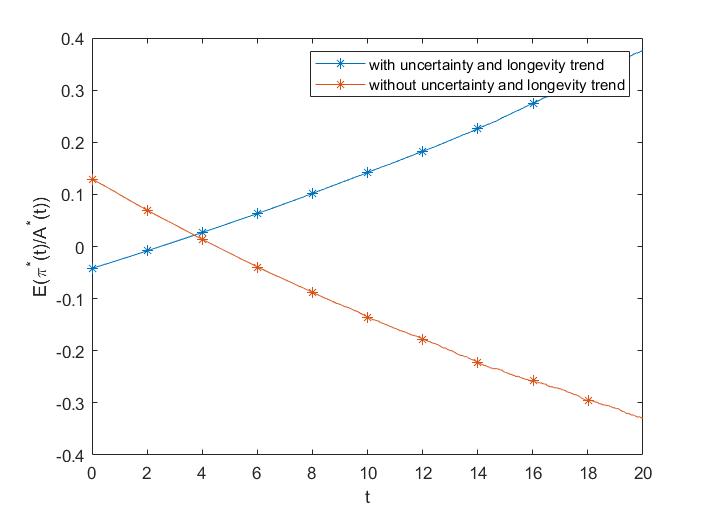}
	
	\caption{The impacts of uncertainty and longevity trend on optimal investment proportion $\mathbb{E}(\frac{\pi^{*}(t)}{A^{*}(t)})$ } \label{pi_longevity_uncertainty}
\end{figure}

\section{Conclusion}
This paper investigates the optimal contribution, the optimal benefit payment and the optimal asset allocation policies to minimize the cost of unstable contribution risk, the cost of unstable benefit risk and discontinuity risk in a hybrid pension plan under model uncertainty. Considering that all the participants are in an aggregate pension fund with both active members and retirees in an overlapping generations economy, we propose a dynamic and controllable model for this hybrid pension fund. Besides, a linear fashion of maximum survival age and an age and time-dependent force of mortality is considered to capture the longevity trend which is decrease with time and increase with age. The fund asset is invested in a risk-free asset and risky asset, and we employ the penalized quadratic deviations as the cost function. By adjusting both the contribution rate and benift according to the surplus or deplict (the difference between the fund asset and the target liabilities), the fund residual can be shared across generations, including employees and retirees. Solving the minimization problem by HJB under the worst scenario, we obtain the robust optimal strategies in our hybrid pension plan. In the part of numercial illustrations, we investigate the impact of different parameters. In most cases, there are opposite trends between the contribution rate and the benefit, since the pension fund should always be stable and balanced. Include, the time interval have a positive impact on contribution and risky aseet proprotion, which means we must have sufficient money to face the long-term hybrid pension system. Besides, Postponing the retirement age could offset the negative impact of decreasing new entrants growth rate, which lead to lower contribution and higher benefit. Therefore, postponing the retirement is an feasible measure to alleviate the aderverse influence of the aging problem. In the section of three special cases, we notice that taking the uncertainty and longevity into consideration increase the proportion of investment to the risk asset and contributions, and decrease the benefits.

However, there are still many aspects that can be researched in our future study. Firstly, we assume that the force of mortality is age and time-dependent. This hypothesis is reasonable due to the decreasing fertility and increasing life expectancy. Furthermore, it is worth mentioning that some research has emphasized the relationship between socio-economic status, income and mortality; see, \cite{Andr2014On} and \cite{2017NDC}. Taking the socio-economic status and income into consideration in our model of mortality might provide practically more equitable risk sharing strategies. Secondly, other sources of persistent systematic risks and market incompleteness, like inflation risk and real interest rate risk might further strengthen the welfare-enhancing potential of intergenerational risk sharing. Since these risks are often non-hedgable in the financial markets. What's more, we assume that the salary is constant 1 whereas it is unrealistic. Actually, it is more reasonal if there is a stochastic salary process, and we will leave it into our future work. At last, by default, We denote the rate of discount by risk-free interest rate. However, \cite{Jesus2009Non} and \cite{2020Time} introduced the analysis of dynamic optimization problems in continuous time setting with non-constant rate of discount. That is a very practical assumption because the non-constant discounting reflect the temporal preferences depend on the temporal position of the decision maker, but it could arise the time inconsistency, so we need modified HJB equation to solve this kind of problem.

\section{Acknowledment}
The authors would be very grateful to referees for their suggestions and this research was supported by the National Natural Science Foundation of China (Grant Nos. 11871052, 12171360, 11771329), Key Project of the Nation Social Science Foundation of China (Grant No. 21AZD071) and Natural Science Foundation of Tianjin City (Grant No. 20JCYBJC01160).


\end{document}